\documentclass[12pt,reqno]{amsart}

\usepackage{times}
\usepackage{graphicx}
\graphicspath{ {./images/} }

\usepackage{hyperref}

\usepackage{amsmath,amsthm,amsfonts,amscd,amssymb}
\usepackage{pst-node}
\usepackage{pst-plot}
\usepackage[all]{xy}
\usepackage{stackengine}

\usepackage{tikz, braids}
\usetikzlibrary{decorations.pathreplacing, arrows}

\usepackage{enumitem,kantlipsum}

\newtheorem{theorem}{Theorem}[section]

\newtheorem{lemma}[theorem]{Lemma}

\newtheorem{notation}[theorem]{Notation}

\newtheorem{question}{Question}
\newtheorem*{theorem*}{Theorem 2.3 (1)}
\newtheorem*{theorem**}{Theorem 2.3 (2)}

\numberwithin{theorem}{section}
\numberwithin{equation}{section}

\newcommand{\la}{\langle}
\newcommand{\ra}{\rangle}

\newcommand{\Comp}{\mathbb{C}}

\newcommand{\n}{\mathbb{N}}

\newcommand{\z}{\mathbb{Z}}

\DeclareRobustCommand{\rGamma}{\text{\reflectbox{$\Gamma$}}}

\begin{document}
\title[{ A central limit theorem for partial transposes}]{A central limit theorem for partial transposes of multipartite Wishart matrices}

\author{Gyunam Park}
\address{Gyunam Park, 
Department of Mathematics Education, Seoul National University, 
Gwanak-Ro 1, Gwanak-Gu, Seoul 08826, Republic of Korea}
\email{pgn518@snu.ac.kr }

\author{Sang-Gyun Youn}
\address{Sang-Gyun Youn, 
Department of Mathematics Education, Seoul National University, 
Gwanak-Ro 1, Gwanak-Gu, Seoul 08826, Republic of Korea}
\email{s.youn@snu.ac.kr }

\maketitle

\begin{abstract}
The partial transposition from quantum information theory provides a new source to distill the so-called asymptotic freeness without the assumption of classical independence between random matrices. Indeed, a recent paper \cite{MP19} established asymptotic freeness between partial transposes in the bipartite situation. In this paper, we prove almost sure asymptotic freeness in the general multipartite situation and establish a central limit theorem for the partial transposes.
\end{abstract}

\section{Introduction}

The origin of free probability theory can be traced back to Voiculescu's works around 1985, and one of the key discoveries is the so-called {\it asymptotic freeness} of independent random matrices with Gaussian entries \cite{Vo91}. This phenomenon extends beyond the Gaussian models and applies to various other models as well. Amongst them are non-Gaussian Wigner matrices \cite{Dy93}, independent Haar unitary random matrices \cite{Vo98} and random permutation matrices \cite{Ni93}, etc.

It is worth noting that all the results mentioned above are assuming independence between the random matrices, resulting in the phenomenon of asymptotic freeness. A natural question arising from this perspective is whether there are fundamentally different approaches to obtaining asymptotic freeness. A positive answer to this question is obtained from the partial transposition \cite{MP19}, which plays a crucial role in quantum information theory (QIT). Indeed, partial transposition is crucial in the problem of entanglement of quantum states and quantum channels \cite{Pe96,HHH96,Wo76,Ch80}, as well as in computing the transmission rate of information \cite{Sh02,SmSm12}, PPT$^2$ conjecture \cite{Ch12,KMP18,RJP18,CMW19,CYT19} and so forth.

This paper focuses on partial transposes of Wishart random matrices, which arise naturally in the context of QIT since the normalizations of Wishart matrices are standard models for {\it random quantum states} \cite{Br96,Ha98,ZS01,SZ04}. An important recent discovery is the asymptotic freeness between partial transposes of a Wishart matrix in the bipartite situation \cite{MP19}. Let $d_1,d_2$ and $p$ be natural numbers, and let $G_{d_1d_2,p}$ be a $d_1d_2\times p$ random matrix with independent complex Gaussian random variables whose mean and variance are $0$ and $1$ respectively. Then the Wishart matrix $W_{d_1d_2,p}$ is given by
\begin{equation}
    W_{d_1d_2,p}=\frac{1}{d_1d_2}G_{d_1d_2,p}G_{d_1d_2,p}^* \in M_{d_1d_2}(\Comp) .
\end{equation}
Let us denote by $T_d$ or simply by $T$ the transpose map $A\mapsto A^t$ on $M_d(\Comp)$ if there is no possibility of confusion. Then the partial transposes of $W_{d_1d_2,p}\in M_{d_1d_2}(\Comp)\cong M_{d_1}(\Comp)\otimes M_{d_2}(\Comp)$ in the bipartite situation are given by
\begin{equation}
    \left\{\begin{array}{llll}

    \vspace{3pt}
    
    W_{d_1d_2,p}&=(\text{id}_{d_1}\otimes \text{id}_{d_2})(W_{d_1d_2,p}),\\

\vspace{3pt}
    
    W_{d_1d_2,p}^{\Gamma}&=(\text{id}_{d_1}\otimes T_{d_2})(W_{d_1d_2,p}),\\

    \vspace{3pt}
    W_{d_1d_2,p}^{\rGamma}&=(T_{d_1}\otimes \text{id}_{d_2})(W_{d_1d_2,p}),\\

\vspace{3pt}
    
    W_{d_1d_2,p}^t &=(T_{d_1}\otimes T_{d_2})(W_{d_1d_2,p}).
    \end{array} \right .
\end{equation}

One of the main results of \cite{MP19} is that the family
\begin{equation}
    \left\{  W_{d_1d_2,p}, W_{d_1d_2,p}^{\Gamma}, W_{d_1d_2,p}^{\rGamma}, W_{d_1d_2,p}^t\right\}
\end{equation}
is asymptotically free under the assumption $\lim d_1=\infty=\lim d_2$ with $\displaystyle \lim \frac{p}{d_1d_2}=c\in (0,\infty)$. From QIT perspective, it is natural to consider a multipartite scenario of quantum communication. Indeed, the bipartite setting is the standard framework to model interactions between two parties, and it is standard to use a multi-fold tensor product to describe possible interactions between multiple parties. In the general $n$-partite situation, we have $2^n$ types of partial transposes of 
\begin{equation}
    W_{d_1\cdots d_n,p}=\frac{1}{d_1\cdots d_n}G_{d_1\cdots d_n,p}G_{d_1\cdots d_n,p}^*\in  M_{d_1}(\Comp)\otimes \cdots \otimes M_{d_n}(\Comp),
\end{equation}
given by
\begin{equation}
    W_{d_1\cdots d_n,p}^{\sigma}=(T_{d_1}^{\sigma_1}\otimes  \cdots \otimes T_{d_n}^{\sigma_n})(W_{d_1\cdots d_n,p}),
\end{equation}
where $\sigma=(\sigma_1,\sigma_2,\cdots,\sigma_n)$ is an arbitrary element of $\left\{0,1\right\}^n$. This paper focuses on two research questions for these partial transposes $W^{\sigma}_{d_1\cdots d_n,p}$. The first main question is as follows.
\begin{question} 
Is the family $\left\{W^{\sigma}_{d_1\cdots d_n,p}\right\}_{\sigma\in \left\{0,1\right\}^n}$ of partial transposes asymptotically free in the general $n$-partite situation assuming $ \lim d_j=\infty$ for all $j=1,2,\cdots,n$ with $\lim \frac{p}{d_1\cdots d_n}=c\in (0,\infty)$? What about almost sure asymptotic freeness?
\end{question}

Note that a partial positive answer to the above {\bf Question 1} can be obtained from a recent paper \cite{MP22}. Indeed, \cite[Corollary 4.15]{MP22} provides an asymptotically free family consisting of $2n$ partial transposes out of $2^n$ choices. We establish the positive answer with full generality to this problem in Theorem \ref{thm500} where we prove almost sure asymptotic freeness for the whole family of partial transposes $\left\{W^{\sigma}_{d_1\cdots d_n,p}\right\}_{\sigma\in \left\{0,1\right\}^n}$. Then, an important advantage of this shift to the multipartite setting is that we have a limitless number of asymptotically free partial transposes $W^{\sigma}_{d_1\cdots d_n,p}$, so it becomes possible to discuss the following problem.

\begin{question}
If the partial transposes are asymptotically free, then is it possible to establish a natural analogue of the central limit theorem?
\end{question}

To do this, in Section \ref{sec-CLT}, we denote by ${\bf d}=(d_1,d_2,\cdots,d_n)$, regard $p=p({\bf d})$ and $n=n({\bf d})$ as functions of ${\bf d}$, and consider the following averages after centering
\begin{equation}
    s_{\bf d}=\frac{1}{|B_{\bf d}|^{\frac{1}{2}}}\sum_{\sigma\in B_{\bf d}} \left ( W_{{\bf d},p}^{\sigma}-c\cdot \text{Id}_{\bf d} \right )
\end{equation}
for certain subsets $B_{\bf d}\subseteq \left\{0,1\right\}^n$ with $\lim |B_{\bf d}|=\infty$. Then, we prove that $(s_{\bf d})_{\bf d}$ converges in moments to the semicircular element of the mean $0$ and the variance $c$, i.e.
\begin{align}
\lim (\mathbb{E}\otimes \text{tr})(s_{\bf d}^{m})=\int_{[-2c,2c]}\frac{t^m}{2\pi c^2}\sqrt{4c^2-t^2}dt
\end{align}
if ${\displaystyle \lim }~ |B_{\bf d}|^m \left ( \frac{1}{\mu(\bf d)}+\left | \frac{p}{d_1d_2\cdots d_n}-c\right | \right )=0$ for all natural numbers $m$ (Theorem \ref{pCLT}), where $\displaystyle \mu({\bf d})=\min_{1\leq j\leq n}d_j$.

\section{Asymptotic freeness of partial transposes}

Let us begin with generalizing some notations and terminologies in \cite{MP19} to the multipartite setting. Let $p$ and $d_1, d_2, \cdots, d_{n}$ be natural numbers with $n\geq 2$, and let 
\begin{equation}
    G=G_{d_1d_2,\cdots d_{n},p}
\end{equation}
be a $d_1 d_2\cdots d_{n}\times p$ random matrix whose entries are independent complex Gaussian random variables with mean $0$ and variance $1$. We denote by $[d]=\left\{1,2,\cdots,d\right\}$ for any natural number $d$, and $[d_1d_2\cdots d_n]=[d_1]\times [d_2]\times \cdots \times [d_n]$ for simplicity. Using a canonical linear isomorphism
\begin{equation}
    M_{d_1d_2\cdots d_{n},p}(\Comp)\cong \Comp^{d_1d_2\cdots d_{n-1}}\otimes M_{d_{n},p}(\Comp),
\end{equation}
let us write $G=\displaystyle \sum_{{\bf i}\in [d_1d_2\cdots d_{n-1}]}e_{\bf i}\otimes G_{\bf i}$ and $\displaystyle G_{\bf i}=\sum_{x=1}^{d_{n}}\sum_{y=1}^p g^{(\bf i)}_{x,y}e_{x,y}\in M_{d_{n},p}(\Comp)$. Then the Wishart matrix $W=\frac{1}{d_1 \cdots d_{n}}GG^*\in M_{d_1d_2\cdots d_{n}}(\Comp)$ is given by
\begin{align}
    &\frac{1}{d_1 \cdots d_{n}}\sum_{\mathbf{i}, \mathbf{j}\in [d_1d_2\cdots d_{n-1}]}e_{\mathbf{i}_1, \mathbf{j}_1} \otimes e_{\mathbf{i}_2, \mathbf{j}_2} \otimes \cdots \otimes e_{\mathbf{i}_{n-1}, \mathbf{j}_{n-1}} \otimes G_\mathbf{i}G_\mathbf{j}^{*}.
\end{align}

Note that we should consider $2^n$-types of partial transpositions of $W$ in the $n$-partite situation.  For any $\sigma=(\sigma_1, \sigma_2, \cdots, \sigma_{n}) \in \{ 0, 1 \}^{n}$, we define the associated partial transposition
\begin{equation}
W^{\sigma}= \left(T_1^{\sigma_1}\otimes T_2^{\sigma_2}\otimes \cdots \otimes T_{n}^{\sigma_{n}}\right)(W)\in M_{d_1d_2\cdots d_{n}}(\Comp)
\end{equation}
where $T_i$ is the transpose operator on each $M_{d_{i}}(\Comp)$.

To compute (non-commutative) joint moments of the partial transposes, let us consider a $\z_2$-valued $m\times n$ matrix $\epsilon=(\epsilon_{ij})_{i\in [m], j \in [n]}$. Then there exist $m$ rows sequences $\epsilon_i=(\epsilon_{ij})_{j=1}^{n}\in \left\{0,1\right\}^{n}$ and their associated partial transposes are given by
\begin{equation}
    W^{\epsilon_1}, W^{\epsilon_2}, \cdots, W^{\epsilon_m}.
\end{equation} 

\subsection{Joint moments of partial transposes}

 In this section, we discuss the $k$-th moments $\mathbb{E}(X_{\epsilon}^k)$ of the following random variable
\begin{align}
X_{\epsilon}= \text{tr} (W^{\epsilon_1} W^{\epsilon_2} \cdots W^{\epsilon_m})=\frac{1}{d_1d_2\cdots d_{n}}\text{Tr} (W^{\epsilon_1} W^{\epsilon_2} \cdots W^{\epsilon_m}).
\end{align}
Here, $\text{tr}=\frac{1}{d}\text{Tr}$ is the normalized trace on $M_d(\Comp)$ and $\epsilon=(\epsilon_{ij})_{i\in [m],j\in [n]}$ is a $\z_2$-valued $m\times n$ matrix with $\epsilon_i=(\epsilon_{ij})_{j=1}^n\in \left\{0,1\right\}^n$. It is unclear whether $X_{\epsilon}$ is a real-valued random variable for now, but it will be explained later in Appendix A.

Recall that \cite[Theorem 3.7]{MP19} covers the case $(n,k)=(2,1)$, and our focus is about the general cases of $(n,k)$. For any natural numbers $k$ and $m$, let us introduce some elementary permutations on 
\begin{equation} 
[\pm km]=[km]\cup [-km]=\left\{1,2,\cdots,km\right\}\cup \left\{-1,-2,\cdots,-km\right\}
\end{equation}
as follows. Recall that the following permutations 
\begin{align}
    \Delta&=(1,-1)\circ (2,-2)\circ \cdots \circ (m,-m)\\
    \Gamma&=(1,2,\cdots,m)
\end{align}
on $[\pm m]$ were introduced in \cite{MP19} to prove asymptotic freeness of partial transposes in the bipartite situation. We define their natural extensions $\Delta^{(k)}$ and $\Gamma^{(k)}$ on $[\pm km]\cong [k]\times [\pm m]$ as the product maps 
\begin{align}
\Delta^{(k)}=&\text{id}_k\times \Delta,\\
\Gamma^{(k)}=&\text{id}_k\times \Gamma.
\end{align}
Then, it is immediate to see that their cycle decompositions on $[\pm km]$ are given by
\begin{align}
\Delta^{(k)}=&(1,-1)\circ (2,-2)\circ \cdots \circ (km, -km),\\
\Gamma^{(k)}=&(1,\cdots,m)(m+1,\cdots,2m)\cdots((k-1)m+1,\cdots,km).
\end{align}

While the $m$ row sequences $\epsilon_1,\cdots,\epsilon_m$ of $(\epsilon_{ij})_{i\in [m],j\in [n]}$ were used to describe multiple partial transposes $W^{\epsilon_1},W^{\epsilon_2},\cdots,W^{\epsilon_m}$, let us use the $j$-th column $\epsilon_j'=(\epsilon_{ij})_{i\in [m]}\in \left\{0,1\right\}^m$ to define a permutation $\mathcal{E}_j$ on $[\pm m]$ by 

\begin{equation}
    \mathcal{E}_j(x)=\left\{ \begin{array}{ll}
x&\text{if }\epsilon_{|x|j}=0\\
-x&\text{if }\epsilon_{|x|j}=1\end{array} \right . .
\end{equation}
Additionally, $\mathcal{E}_j$ extends to a permutation 
\begin{equation}
\mathcal{E}^{(k)}_j=\text{id}_k\times \mathcal{E}_j:[k]\times [\pm m]\rightarrow [k]\times [\pm m]    
\end{equation} 
given by $\mathcal{E}^{(k)}_j(s,s') = (s,\mathcal{E}_j(s'))$. 

Now, we are ready to provide an explicit formula for the following $k$-th moments
\begin{equation}
   \mathbb{E}\left ( X_{\epsilon}^k \right )=\mathbb{E}\left ( \left [ \text{tr}(W^{\epsilon_1}W^{\epsilon_2}\cdots W^{\epsilon_m}) \right ]^k \right ),
\end{equation}
generalizing \cite[Theorem 3.7]{MP19} with full generality under the following notations.

\begin{notation}
Note that any permutation $\sigma\in S_m$ is associated to a partition $\pi$ of $[m]$ using the cyclic decomposition of $\sigma$. We denote by $\sharp(\sigma)$ the number of blocks of $\pi$, and denote by $\pi \vee \pi'$ the supremum of two partitions $\pi$ and $\pi'$. When we regard $\sigma\in S_m$ as a permutation on $[\pm m]$, the extension is considered the identity function on $[-m]=\left\{-m,\cdots,-1\right\}$.
\end{notation}

Our proof for the following theorem is systematic but requires heavy use of notations, so let us present the proof separately in Appendix A.

\begin{theorem}\label{thm302}
    Let $\epsilon=(\epsilon_{ij})_{i\in [m],j\in [n]}$ be a $\z_2$-valued $m\times n$ matrix with $\epsilon_i=(\epsilon_{ij})_{j=1}^{n}\in \left\{0,1\right\}^{n}$ for all $i\in [m]$, and let $X_{\epsilon}=\text{tr}(W^{\epsilon_1}W^{\epsilon_2}\cdots W^{\epsilon_m})$. Then, for any natural number $k$, we have
    \begin{equation}\label{eq34}
    \mathbb{E}(X_{\epsilon}^k) = \sum_{\sigma \in S_{km}} \left( \frac{p}{d_1 \cdots d_{n}} \right)^{\sharp(\sigma)} \prod_{j=1}^n d_j^{f_{k,j}(\epsilon,\sigma)},  
    \end{equation}
    where the exponent $f_{k,j}(\epsilon,\sigma)$ is given by
    \begin{equation}
        \sharp(\mathcal{E}^{(k)}_{j} \Gamma^{(k)} \Delta^{(k)} (\Gamma^{(k)})^{-1} \mathcal{E}^{(k)}_{j} \vee \sigma \Delta^{(k)} \sigma^{-1}) + \sharp(\sigma) - k(m+1)
    \end{equation}
    for all $\sigma\in S_{km}$ and $j\in [n]$.
\end{theorem}

Recall that we have 
\begin{equation}\label{lem303}
2\cdot \sharp(\pi_1 \vee \pi_2) = \sharp(\pi_1 \circ \pi_2)
\end{equation}
for any pairings $\pi_{1}, \pi_{2} \in \mathcal{P}_2(n)$ by \cite[Lemma 2]{MP13}, and both the permutations $\mathcal{E}^{(k)}_{j} \Gamma^{(k)} \Delta^{(k)} (\Gamma^{(k)})^{-1} \mathcal{E}^{(k)}_{j}$ and $\sigma \Delta^{(k)} \sigma^{-1}$ are indeed pairings. Thus, our main focus from now is to analyze
\begin{align}
    &2\cdot \sharp(\mathcal{E}^{(k)}_{j} \Gamma^{(k)} \Delta^{(k)} (\Gamma^{(k)})^{-1} \mathcal{E}^{(k)}_{j} \vee \sigma \Delta^{(k)} \sigma^{-1})\\
    &=\sharp(\mathcal{E}^{(k)}_{j} \Gamma^{(k)} \Delta^{(k)} (\Gamma^{(k)})^{-1} \mathcal{E}^{(k)}_{j} \sigma \Delta^{(k)} \sigma^{-1})\\
    &=\sharp(\Gamma^{(k)} \Delta^{(k)} (\Gamma^{(k)})^{-1} \Delta^{(k)} \mathcal{E}^{(k)}_{j} \Delta^{(k)} \sigma \Delta^{(k)} \sigma^{-1} \mathcal{E}^{(k)}_{j}).
\end{align}

\subsection{Almost sure asymptotic freeness in the multipartite setting}

To establish the almost sure asymptotic freeness, our main technical question is how to compute the exponents $f_{k,j}(\epsilon,\sigma)$. Let us write $f_j=f_{1,j}$ for simplicity if there is no possibility of confusion. Recall that the case $(n,k)=(2,1)$ was studied in \cite{MP19} for the bipartite situation. To consider the general cases of $(n,k)$, it is necessary to develop a new framework to study the general situation $k\geq 2$. 

Let us consider the following family of sets
\begin{equation}
    \mathcal{A}_k := \{ A_1,A_2,\cdots,A_k \},
\end{equation}
for general $k$, where $A_j= [jm]\setminus [(j-1)m]=\left\{(j-1)m+1,\cdots, jm\right\}$. We also denote by $\la \mathcal{A}_k \ra := \{ \cup_{A \in \mathcal{S}}A : \mathcal{S} \subseteq \mathcal{A}_k \}$. Then the main theorem of this section is stated as follows.

\begin{theorem}\label{thm-main1}
 Let $\sigma \in S_{km}$ and let $\epsilon=(\epsilon_{ij})_{i\in [m],j\in [n]}$ be a $\z_2$-valued $m\times n$ matrix.
    \begin{enumerate}
        \item Assume that $k\geq 2$ and there exist non-empty disjoint subsets $C_1\in \la\mathcal{A}_k\ra$ and $C_2\in \la\mathcal{A}_k\ra$ such that $\sigma(C_1)=C_1$, $\sigma(C_2)=C_2$ and $[km]=C_1\cup C_2$. Let $|C_1|=k_1m$ and $|C_2|=k_2m$ and consider bijective increasing functions $c_1 : [k_1m] \rightarrow C_1$ and $c_2 : [k_2m] \rightarrow C_2$. Then we have
    \begin{align}
        &\sharp\left (\mathcal{E}^{(k)}_{j} \Gamma^{(k)} \Delta^{(k)} (\Gamma^{(k)})^{-1} \mathcal{E}^{(k)}_{j} \vee \sigma \Delta^{(k)} \sigma^{-1}\right )\\
        &=\sum_{i=1,2}\sharp\left (\mathcal{E}^{(k_i)}_{j} \Gamma^{(k_i)} \Delta^{(k_i)} (\Gamma^{(k_i)})^{-1} \mathcal{E}^{(k_i)}_{j} \vee (c_i ^{-1} \sigma c_i) \Delta^{(k_i)} (c_i ^{-1} \sigma c_i)^{-1}\right )
    \end{align}
    for all $j\in [n]$. In particular, we have
    \begin{equation}
        f_{k,j}(\epsilon,\sigma)=f_{k_1,j}(\epsilon,c_1^{-1}\sigma c_1)+f_{k_2,j}(\epsilon,c_2^{-1}\sigma c_2).
    \end{equation}
    \item  Assume that $k\geq 2$ and there are no non-empty disjoint subsets $C_1\in \la \mathcal{A}_k\ra$ and $C_2\in \la\mathcal{A}_k\ra$ such that $\sigma(C_1)=C_1$, $\sigma(C_2)=C_2$ and $[km]=C_1\cup C_2$. Then we have 
   \begin{equation}
   f_{k,j}(\epsilon,\sigma)\leq 2-2k\leq -2
   \end{equation}
   for all $j\in [n]$. 
    \end{enumerate}
\end{theorem}

A proof of the above Theorem \ref{thm-main1} will be presented in the next subsection \ref{sec-proof}. In this section, let us focus on how this result is applied to prove almost sure asymptotic freeness of the partial transposes $\left\{W^{\sigma}\right\}_{\sigma\in \left\{0,1\right\}^n}$. To proceed, let us recall an important lemma from \cite{MP19}. For a function $x:[m] \rightarrow Y$, let us denote by
\begin{equation}
    \ker(x)=\{x^{-1}(t):t\in Y\} \backslash \{\emptyset\}.
\end{equation}

\begin{lemma}

 Let $\sigma \in S_{m}$ and let $\epsilon=(\epsilon_{ij})_{i\in [m],j\in [n]}$ be a $\z_2$-valued $m\times n$ matrix with $\epsilon_j'=(\epsilon_{ij})_{i\in [m]}\in \left\{0,1\right\}^m$.
 
\begin{enumerate} \label{lem400}
    \item
    Then $f_j(\epsilon,\sigma)=f_{1,j}(\epsilon,\sigma)<0$ holds unless $\epsilon_j'$ is constant on the cycles of $\sigma$.
    \item \label{lem402}
   If $\epsilon_j'$ is constant on the cycles of $\sigma$, then $f_j(\epsilon,\sigma) \le 0$ with equality holds precisely when the associated partition of $\sigma$ is non-crossing.
\end{enumerate}
In particular, if $f_j(\epsilon,\sigma)\equiv 0$ for all $j\in [n]$ and if $\pi$ is the associated partition of $\sigma$, then $\text{ker}(\epsilon)\geq \pi$ holds, i.e. each block of $\pi$ is contained in a block of $\text{ker}(\epsilon)$.
\end{lemma}

Then, applying Theorem \ref{thm-main1} with Lemma \ref{lem400}, we reach the following almost sure asymptotic freeness for the general cases of $(n,k)$.

\begin{theorem}\label{thm500}
    Suppose that $\displaystyle \lim d_j =\infty$ for all $j\in [n]$ with the condition $\displaystyle \lim \frac{p}{d_{1} \cdots d_{n}} = c\in (0,\infty)$. Then the family $\{ W^\sigma \}_{\sigma \in \{ 0, 1 \}^{n}}$ of the partial transposes is almost surely asymptotically free.
\end{theorem}

\begin{proof}
As the first step, let us prove asymptotic freeness by showing that all the mixed cumulants vanish as in \cite{MP19}. Recall that the joint moment $(\mathbb{E}\otimes \text{tr})(W^{\epsilon_1}W^{\epsilon_2}\cdots W^{\epsilon_m})$ is given by 
    \begin{align}\label{eq205}
    &\sum_{\sigma \in S_{m}} \left( \frac{p}{d_1 \cdots d_{n}} \right)^{\sharp(\sigma)} \prod_{j=1}^n d_j^{f_{j}(\epsilon,\sigma)}
    \end{align} 
    for any $\z_2$-valued $m\times n$ matrix $\epsilon=(\epsilon_{ij})_{i\in [m],j\in [n]}$ by Theorem \ref{thm302}. Furthermore, Lemma \ref{lem400} tells us that
    \begin{equation}
        \lim \left [ \left( \frac{p}{d_1 \cdots d_{n}} \right)^{\sharp(\sigma)} \prod_{j=1}^n d_j^{f_{j}(\epsilon,\sigma)} \right ] = c^{\sharp(\sigma)}\cdot 0=0
    \end{equation}
 if the associated partition of $\sigma \in S_{m}$ is crossing, so it is enough to consider only the cases where the associated partitions are non-crossing in \eqref{eq205}. Let us denote by $NC(m)$ the set of all non-crossing partitions on $[m]$ and by $S(m,\pi)$ the set of all permutations $\sigma\in S_m$ whose associated partition is $\pi$. Then we have
    \begin{align}
        &\lim (\mathbb{E}\otimes \text{tr})(W^{\epsilon_1}W^{\epsilon_2}\cdots W^{\epsilon_m})\\
        \label{cumulant}&=\sum_{\pi \in NC(m)} c^{\sharp(\pi)} \sum_{\sigma\in S(m,\pi)} \lim \prod_{j=1}^n d_j^{f_{j}(\epsilon,\sigma)},
    \end{align}
and Lemma \ref{lem400} implies
\begin{align}
     \lim \prod_{j=1}^n d_j^{f_{j}(\epsilon,\sigma)}=\left\{\begin{array}{ll}
     1&\text{if }\text{ker}(\epsilon) \ge \pi\\
     0&\text{otherwise}
     \end{array} \right .
\end{align}
for all $\sigma\in S(m,\pi)$ where we regard $i\mapsto \epsilon_i$ as a function from $[m]$ into $\left\{0,1\right\}^n$. Let $V_1,V_2,\cdots,V_r$ be the disjoint block decomposition of $\pi\in NC(m)$, and write
\begin{equation}\label{eq206}
    \delta_{T}(W^{\epsilon_1},W^{\epsilon_2},\cdots, W^{\epsilon_m})=\left\{\begin{array}{ll}
    1&\text{if }\epsilon_{t_1}=\epsilon_{t_2}=\cdots= \epsilon_{t_l}\\
    0&\text{otherwise}
    \end{array} \right . .
\end{equation}
for any subset $T=\left\{t_1,t_2,\cdots,t_l\right\}\subseteq [m]$ with $t_1<t_2<\cdots<t_l$. Then \eqref{cumulant} can be written as
\begin{equation}
    \sum_{\pi \in NC(m)} \prod_{i=1}^r  c \cdot (|V_i|-1)! \cdot \delta_{V_i}(W^{\epsilon_1},W^{\epsilon_2},\cdots, W^{\epsilon_m}).
\end{equation}
A crucial step here is to note that 
\begin{align}
   \prod_{i=1}^r  c \cdot (|V_i|-1)! \cdot  \delta_{V_i}(W^{\epsilon_1},W^{\epsilon_2},\cdots, W^{\epsilon_m})
\end{align}
coincides with namely the {\it free cumulant} 
\begin{equation}
\kappa_{\pi}(W^{\epsilon_1},W^{\epsilon_2},\cdots ,W^{\epsilon_m})=\prod_{i=1}^r \kappa_{V_i}(W^{\epsilon_1},W^{\epsilon_2},\cdots ,W^{\epsilon_m}).
\end{equation}
Thus, the above \eqref{eq206} tells us that all mixed cumulants vanish, and this fact allows us to conclude that the given family $\left\{W^{\sigma}\right\}$ is asymptotically free by \cite[Theorem 16]{MS17} or \cite[Theorem 11.20]{NS06}.

Now, our second step is to prove
\begin{equation}
    \text{Var}(X_{\epsilon})=O(d_1^{-2}\cdots d_n^{-2})
\end{equation}
to establish almost sure asymptotic freeness. Note that the following identity
\begin{align}
        \mathbb{E}(X_{\epsilon})^{2} = \underset{\sigma([m])=[m]}{\sum_{\sigma \in S_{2m}}} \left( \frac{p}{d_{1} \cdots d_{n}} \right)^{\sharp(\sigma)} \prod_{j=1}^{n} d_{j}^{f_{2,j}(\epsilon,\sigma)}
    \end{align}
is a direct consequence from Theorem \ref{thm-main1} (1). Indeed, for any $\sigma \in S_{2m}$ such that $\sigma([m])=[m]$ and $j\in[n]$, we have 
    \begin{align}\label{eqn501}
        \begin{split}
            &\sharp(\mathcal{E}^{(2)}_{j} \Gamma^{(2)} \Delta^{(2)} (\Gamma^{(2)})^{-1} \mathcal{E}^{(2)}_{j} \vee \sigma \Delta^{(2)} \sigma^{-1})\\
            &=\sum_{i=1,2}\sharp(\mathcal{E}_{j} \Gamma \Delta \Gamma^{-1} \mathcal{E}_{j} \vee (c_i ^{-1} \sigma c_i) \Delta (c_i ^{-1} \sigma c_i)^{-1})
        \end{split}
    \end{align}
and $f_{2,j}(\epsilon,\sigma)=f_{1,j}(\epsilon,c_1^{-1}\sigma c_1)+f_{1,j}(\epsilon,c_2^{-1}\sigma c_2)$ by Theorem \ref{thm-main1} (1). Here, $c_{1}:[m]\rightarrow[m]$ is the identity map and $c_{2}:[m]\rightarrow[2m]\backslash[m]$ is given by $c_{2}(i)=m+i$. Furthermore, since $\{ \sigma \in S_{2m}:\sigma([m])=[m] \}$ is naturally identified with $S_{m} \times S_{m}$ via $\sigma \mapsto (c_{1}^{-1} \sigma c_{1}, c_{2}^{-1} \sigma c_{2})$, we can see that
    \begin{align}
        &\underset{\sigma([m])=[m]}{\sum_{\sigma \in S_{2m}}} \left( \frac{p}{d_{1} \cdots d_{n}} \right)^{\sharp(\sigma)} \prod_{j=1}^{n} d_{j}^{f_{2,j}(\epsilon,\sigma)}\\
        &=\underset{\sigma([m])=[m]}{\sum_{\sigma \in S_{2m}}} \left( \frac{p}{d_{1} \cdots d_{n}} \right)^{\sharp(c_{1}^{-1} \sigma c_{1})+\sharp(c_{2}^{-1} \sigma c_{2})} \prod_{j=1}^{n} d_{j}^{f_{1,j}(\epsilon,c_{1}^{-1} \sigma c_{1}) + f_{1,j}(\epsilon,c_{2}^{-1} \sigma c_{2})}\\
        &=\sum_{\tau_{1},\tau_{2} \in S_{m}} \left( \frac{p}{d_{1} \cdots d_{n}} \right)^{\sharp(\tau_{1})+\sharp(\tau_{2})} \prod_{j_1=1}^{n} d_{j_1}^{f_{1,j_1}(\epsilon,\tau_{1})} \prod_{j_2=1}^{n} d_{j_2}^{f_{1,j_2}(\epsilon,\tau_{2})} =\mathbb{E}(X_{\epsilon})^{2}.
    \end{align}

Then Theorem \ref{thm302} and Theorem \ref{thm-main1} (2) tell us that
    \begin{align}
 \text{Var}(X_{\epsilon})&=\mathbb{E}(X_{\epsilon}^{2})-\mathbb{E}(X_{\epsilon})^{2}\\
=&\underset{\sigma([m])\ne[m]}{\sum_{\sigma \in S_{2m}}} \left( \frac{p}{d_{1} \cdots d_{n}} \right)^{\sharp(\sigma)} \prod_{j=1}^{n} d_{j}^{f_{2,j}(\epsilon,\sigma)}
    \end{align}
with $f_{2,j}(\epsilon,\sigma)\leq -2$ for all $j\in [n]$. Finally, since $(\frac{p}{d_1d_2\cdots d_n})_{n\in \n}$ has a uniform upper bound $M>1$ from the assumption, we can conclude that
    \begin{align}
        \text{Var}(X_{\epsilon})\leq (2m)!M^{2m}(d_1d_2\cdots d_n)^{-2}.
    \end{align}

\end{proof}


\subsection{Proof of Theorem \ref{thm-main1}}\label{sec-proof}

Let $\sigma\in S_{km}$ and let $\epsilon=(\epsilon_{ij})_{i\in [m],j\in [n]}$ be a $\z_2$-valued $m\times n$ matrix. Let us begin with a proof of the first part of Theorem \ref{thm-main1}.

\begin{theorem*}
    Let $\sigma \in S_{km}$ with $k\geq 2$ and suppose that there exist non-empty disjoint subsets $C_1\in \la\mathcal{A}_k\ra$ and $C_2\in \la\mathcal{A}_k\ra$ such that $\sigma(C_1)=C_1$, $\sigma(C_2)=C_2$ and $[km]=C_1\cup C_2$. Let $|C_1|=k_1m$ and $|C_2|=k_2m$ and consider bijective increasing functions $c_1 : [k_1m] \rightarrow C_1$ and $c_2 : [k_2m] \rightarrow C_2$. Then we have
    \begin{align}
        &\sharp\left (\mathcal{E}^{(k)}_{j} \Gamma^{(k)} \Delta^{(k)} (\Gamma^{(k)})^{-1} \mathcal{E}^{(k)}_{j} \vee \sigma \Delta^{(k)} \sigma^{-1}\right )\\
        &=\sum_{i=1,2}\sharp\left (\mathcal{E}^{(k_i)}_{j} \Gamma^{(k_i)} \Delta^{(k_i)} (\Gamma^{(k_i)})^{-1} \mathcal{E}^{(k_i)}_{j} \vee (c_i ^{-1} \sigma c_i) \Delta^{(k_i)} (c_i ^{-1} \sigma c_i)^{-1}\right )
    \end{align}
    for all $j\in [n]$. In particular, we have
    \begin{equation}
        f_{k,j}(\epsilon,\sigma)=f_{k_1,j}(\epsilon,c_1^{-1}\sigma c_1)+f_{k_2,j}(\epsilon,c_2^{-1}\sigma c_2)
    \end{equation}
    for all $j\in [n]$.
\end{theorem*}

\begin{proof}
    Note that we have
    \begin{align}
        &2\cdot \sharp(\mathcal{E}^{(k)}_{j} \Gamma^{(k)} \Delta^{(k)} (\Gamma^{(k)})^{-1} \mathcal{E}^{(k)}_{j} \vee \sigma \Delta^{(k)} \sigma^{-1})\\
        &=\sharp(\mathcal{E}^{(k)}_{j} \Gamma^{(k)} \Delta^{(k)} (\Gamma^{(k)})^{-1} \mathcal{E}^{(k)}_{j} \sigma \Delta^{(k)} \sigma^{-1})
    \end{align}
    for all $j\in [n]$ thanks to \eqref{lem303}, and the given condition $\sigma(C_1)=C_1$ and $\sigma(C_2)=C_2$ implies
    \begin{align}
        \left (\mathcal{E}^{(k)}_{j} \Gamma^{(k)} \Delta^{(k)} (\Gamma^{(k)})^{-1} \mathcal{E}^{(k)}_{j} \sigma \Delta^{(k)} \sigma^{-1}\right ) \left (C_i \cup (-C_i)\right )=C_i \cup (-C_i)
    \end{align}
    for both cases $i=1$ and $i=2$. Thus, we reach the following conclusion
    \begin{align}
    & 2\cdot \sharp(\mathcal{E}^{(k)}_{j} \Gamma^{(k)} \Delta^{(k)} (\Gamma^{(k)})^{-1} \mathcal{E}^{(k)}_{j} \vee \sigma \Delta^{(k)} \sigma^{-1})\\
        &=\sharp(\mathcal{E}^{(k)}_{j} \Gamma^{(k)} \Delta^{(k)} (\Gamma^{(k)})^{-1} \mathcal{E}^{(k)}_{j} \sigma \Delta^{(k)} \sigma^{-1})\\
        &=\sum_{i=1,2}\sharp\left (\mathcal{E}^{(k)}_{j} \Gamma^{(k)} \Delta^{(k)} (\Gamma^{(k)})^{-1} \mathcal{E}^{(k)}_{j} \sigma \Delta^{(k)} \sigma^{-1}|_{C_i \cup (-C_i)}\right )\\
        &=\sum_{i=1,2}\sharp(\mathcal{E}^{(k_i)}_{j} \Gamma^{(k_i)} \Delta^{(k_i)} (\Gamma^{(k_i)})^{-1} \mathcal{E}^{(k_i)}_{j} (c_i ^{-1} \sigma c_i) \Delta^{(k_i)} (c_i ^{-1} \sigma c_i)^{-1}).
    \end{align}
    
    Additionally, the last conclusion is immediate since $\sharp(\sigma)=\sharp(c_1^{-1}\sigma c_1)+\sharp(c_2^{-1}\sigma c_2)$ and $k(m+1)=k_1(m+1)+k_2(m+1)$.
\end{proof}

From now on, let us suppose that $k\geq 2$ and there do not exist non-empty disjoint $C_1\in \la \mathcal{A}_k\ra$ and $C_2\in \la\mathcal{A}_k\ra$ such that $\sigma(C_1)=C_1$, $\sigma(C_2)=C_2$ and $[km]=C_1\cup C_2$. In this case, we can construct a sequence of elements $(x_i)_{i\in [k-1]}$ and a bijective function $\tau:[k]\rightarrow [k]$ such that
\begin{itemize}
\item  $x_i\in A_{\tau(1)}\cup \cdots \cup A_{\tau(i)}$ for all $i \in [k-1]$,
\item $\sigma(x_{i})\in A_{\tau(i+1)}
$ for all $i \in [k-1]$.
\end{itemize}

For two disjoint subsets $S$ and $T$ of $[\pm km]$, let us write $S\sim_{\phi}T$ if there exists an element $x\in S$ such that $\phi(x)\in T$ or an element $y\in T$ such that $\phi(y)\in S$ by a bijective function $\phi$ on $[\pm km]$. Note that $\sim_{\phi}$ is a symmetric relation. 

\begin{lemma}\label{lem208}
    From the above notations, there exist subsets $V_{j,1},V_{j,2},\cdots,V_{j,k}$ of $[\pm km]$ such that
\begin{itemize}
    \item $V_{j,i}$ is one of $A_{\tau(i)}$ and $-A_{\tau(i)}$ for all $i\in [k]$,
    \item $\displaystyle \bigcup_{t\in [i]} V_{j,t}\sim V_{j,i+1}$ for all $i\in [k-1]$ by $\mathcal{E}^{(k)}_{j} \Delta^{(k)} \sigma \Delta^{(k)} \sigma^{-1} \mathcal{E}^{(k)}_{j}$.
\end{itemize}    
\end{lemma}
\begin{proof}
Let us use the above sequence $x_1,x_2,\cdots,x_{k-1}$ to construct the subsets $V_{j,1},V_{j,2},\cdots,V_{j,k}$. Let us start with $V_{j,1}=A_{\tau(1)}$. Then, the following table of direct calculations tells us how to decide $V_{j,i+1}$ from $V_{j,1}, \cdots, V_{j,i}$.

\begin{table}[h!]
  \begin{center}
    \label{tab:table1}
    \begin{tabular}{|l|} 
\hline

(1) $\mathcal{E}^{(k)}_{j} : \begin{array}{ll}x_i \mapsto x_i \\
\sigma(x_i)\mapsto \sigma(x_i) \end{array}$ $\Rightarrow$ $\mathcal{E}^{(k)}_{j} \Delta^{(k)} \sigma \Delta^{(k)} \sigma^{-1} \mathcal{E}^{(k)}_{j}: \begin{array}{ll}\sigma(x_i) \mapsto x_i \\
-x_i\mapsto -\sigma(x_i) \end{array}$ \\

        \hline

(2) $\mathcal{E}^{(k)}_{j} : \begin{array}{ll}x_i \mapsto x_i \\
\sigma(x_i)\mapsto -\sigma(x_i) \end{array}$ $\Rightarrow$ $\mathcal{E}^{(k)}_{j} \Delta^{(k)} \sigma \Delta^{(k)} \sigma^{-1} \mathcal{E}^{(k)}_{j}: \begin{array}{ll}-\sigma(x_i) \mapsto x_i \\
-x_i\mapsto \sigma(x_i) \end{array}$ \\

        \hline

(3) $\mathcal{E}^{(k)}_{j} : \begin{array}{ll}x_i \mapsto -x_i \\
\sigma(x_i)\mapsto \sigma(x_i) \end{array}$ $\Rightarrow$ $\mathcal{E}^{(k)}_{j} \Delta^{(k)} \sigma \Delta^{(k)} \sigma^{-1} \mathcal{E}^{(k)}_{j}: \begin{array}{ll} x_i \mapsto -\sigma(x_i) \\
\sigma(x_i)\mapsto -x_i \end{array}$ \\

        \hline

(4) $\mathcal{E}^{(k)}_{j} : \begin{array}{ll}x_i \mapsto -x_i \\
\sigma(x_i)\mapsto -\sigma(x_i) \end{array}$ $\Rightarrow$ $\mathcal{E}^{(k)}_{j} \Delta^{(k)} \sigma \Delta^{(k)} \sigma^{-1} \mathcal{E}^{(k)}_{j}: \begin{array}{ll} x_i \mapsto \sigma(x_i) \\
-\sigma(x_i)\mapsto -x_i \end{array}$ \\

        \hline
        
    \end{tabular}
  \end{center}
\end{table}

Indeed, if $x_{i}\in A_{\tau(i)}$ and $V_{j,i}=A_{\tau(i)}$ (resp. $V_{j,i}=-A_{\tau(i)}$), then we take $V_{j,i+1}=A_{\tau(i+1)}$ (resp. $V_{j,i+1}=-A_{\tau(i+1)}$) in the first or the fourth cases and take $V_{j,i+1}=-A_{\tau(i+1)}$ (resp. $V_{j,i+1}=A_{\tau(i+1)}$) in the second or the third cases.
\end{proof}

Under the notations above, let us denote by $W_{j,i}=-V_{j,i}$, $V_j=\bigcup_{i\in [k]} V_{j,i}$ and $W_j=\bigcup_{i\in [k]} W_{j,i}$. Then we are ready to prove the second part of Theorem 2.3. Our strategy is to adapt the proof of \cite[Lemma 4.3]{MP19} and to divide the general situation into the case where $V_j\sim W_j$ and the other case where $V_j\nsim W_j$ by $\mathcal{E}^{(k)}_{j} \Delta^{(k)} \sigma \Delta^{(k)} \sigma^{-1} \mathcal{E}^{(k)}_{j}$.

\begin{theorem**}
   Suppose that $k \ge 2$ and there do not exist non-empty disjoint $C_1\in \la \mathcal{A}_k\ra$ and $C_2\in \la\mathcal{A}_k\ra$ such that $\sigma(C_1)=C_1$, $\sigma(C_2)=C_2$ and $[km]=C_1\cup C_2$. Then we have 
   \begin{equation}
   f_{k,j}(\epsilon,\sigma)\leq 2-2k  \leq -2 
   \end{equation}
   for any $\z_2$-valued $m\times n$ matrices $\epsilon=(\epsilon_{ij})_{i\in [m],j\in [n]}$.
\end{theorem**}

\begin{proof}

(Case 1: $V_j\sim W_j$ by $\mathcal{E}^{(k)}_{j} \Delta^{(k)} \sigma \Delta^{(k)} \sigma^{-1} \mathcal{E}^{(k)}_{j}$) In this case, the subgroup generated by $\mathcal{E}^{(k)}_{j} \Delta^{(k)} \sigma \Delta^{(k)} \sigma^{-1} \mathcal{E}^{(k)}_{j}$ and $\Gamma^{(k)} \Delta^{(k)} (\Gamma^{(k)})^{-1} \Delta^{(k)}$ acts on $[\pm km]$ transitively by Lemma \ref{lem208} and the given assumption, so there exists a non-negative integer $g$ satisfying \small
    \begin{align}
        \notag&2km+2(1-g)\\
        \notag&=\sharp(\mathcal{E}^{(k)}_{j} \Delta^{(k)} \sigma \Delta^{(k)} \sigma^{-1} \mathcal{E}^{(k)}_{j})+\sharp(\mathcal{E}^{(k)}_{j} \Delta^{(k)} \sigma \Delta^{(k)} \sigma^{-1} \mathcal{E}^{(k)}_{j} \Gamma^{(k)}\Delta^{(k)}(\Gamma^{(k)})^{-1}\Delta^{(k)})\\
        &{\color{white}ttttttttttttttttttttttttttttttttttttttttttttt}+\sharp(\Gamma^{(k)}\Delta^{(k)}(\Gamma^{(k)})^{-1}\Delta^{(k)})\\
        &=2 \sharp(\sigma)+\sharp(\mathcal{E}^{(k)}_{j} \Delta^{(k)} \sigma \Delta^{(k)} \sigma^{-1} \mathcal{E}^{(k)}_{j} \Gamma^{(k)}\Delta^{(k)}(\Gamma^{(k)})^{-1}\Delta^{(k)})+2k.
    \end{align}
\normalsize See \cite{MS17} for more details about the existence of $g$, and the second equality comes from direct calculations. Then we have
    \begin{align}\label{eqn502}
        &2\cdot \sharp(\mathcal{E}^{(k)}_{j} \Gamma^{(k)} \Delta^{(k)} (\Gamma^{(k)})^{-1} \mathcal{E}^{(k)}_{j} \vee \sigma \Delta^{(k)} \sigma^{-1})\\
        &=\sharp(\mathcal{E}^{(k)}_{j} \Gamma^{(k)} \Delta^{(k)} (\Gamma^{(k)})^{-1} \mathcal{E}^{(k)}_{j} \sigma \Delta^{(k)} \sigma^{-1})\\
        &=\sharp(\Gamma^{(k)} \Delta^{(k)} (\Gamma^{(k)})^{-1} \mathcal{E}^{(k)}_{j} \sigma \Delta^{(k)} \sigma^{-1} \mathcal{E}^{(k)}_{j})\\
        &=\sharp(\Gamma^{(k)} \Delta^{(k)} (\Gamma^{(k)})^{-1} \Delta^{(k)} \mathcal{E}^{(k)}_{j} \Delta^{(k)} \sigma \Delta^{(k)} \sigma^{-1} \mathcal{E}^{(k)}_{j})\\
        &=\sharp(\mathcal{E}^{(k)}_{j} \Delta^{(k)} \sigma \Delta^{(k)} \sigma^{-1} \mathcal{E}^{(k)}_{j} \Gamma^{(k)} \Delta^{(k)} (\Gamma^{(k)})^{-1} \Delta^{(k)})\\
        &=2km+2(1-k-g)-2\sharp(\sigma),
    \end{align}
where we used $\mathcal{E}^{(k)}_j \Delta^{(k)}=\Delta^{(k)} \mathcal{E}^{(k)}_j$ at the third equality. Thus, we obtain 
    \begin{align}
        \notag f_{k,j}(\epsilon,\sigma)&= \sharp(\mathcal{E}^{(k)}_{j} \Gamma^{(k)} \Delta^{(k)} (\Gamma^{(k)})^{-1} \mathcal{E}^{(k)}_{j} \vee \sigma \Delta^{(k)} \sigma^{-1}) + \sharp(\sigma) - k(m+1)\\
        &  = (km+(1-k-g)-\sharp(\sigma))+\sharp(\sigma)-k(m+1) \\
        &=1-2k-g \leq 1-2k\leq 2-2k.
    \end{align}

(Case 2: $V_j\nsim W_j$ by $\mathcal{E}^{(k)}_{j} \Delta^{(k)} \sigma \Delta^{(k)} \sigma^{-1} \mathcal{E}^{(k)}_{j}$) In this case, we have 
\begin{equation}
\mathcal{E}^{(k)}_{j} \Delta^{(k)} \sigma \Delta^{(k)} \sigma^{-1} \mathcal{E}^{(k)}_{j}(V_j)=V_j   
\end{equation} 
and the subgroup generated by $\tau_{1,V_j}=\mathcal{E}^{(k)}_{j} \Delta^{(k)} \sigma \Delta^{(k)} \sigma^{-1} \mathcal{E}^{(k)}_{j}\Bigr|_{V_j}$ and $\tau_{2,V_j}=\Gamma^{(k)}\Delta^{(k)}(\Gamma^{(k)})^{-1}\Delta^{(k)}\Bigr|_{V_j}$ acts on $V_j$ transitively by Lemma \ref{lem208}.
    As in the case $V_j\sim W_j$, there exists a non-negative integer $g$ satisfying
    \begin{align}
    km+2(1-g)&=\sharp(\tau_{1,V_j})+\sharp(\tau_{1,V_j}\circ \tau_{2,V_j})+\sharp(\tau_{2,V_j})\\
    &=\sharp(\tau_{1,V_j})+\sharp(\tau_{1,V_j}\circ \tau_{2,V_j})+k,
    \end{align}
    implying 
\begin{align}
    \sharp(\tau_{1,V_j}\circ \tau_{2,V_j})&=km+2-k-2g-\sharp(\tau_{1,V_j}).
\end{align}

On the other side, we define $\tau_{1,W_{j}}=\mathcal{E}^{(k)}_{j} \Delta^{(k)} \sigma \Delta^{(k)} \sigma^{-1} \mathcal{E}^{(k)}_{j}\Bigr|_{W_j}$ and $\tau_{2,W_j}=\Gamma^{(k)}\Delta^{(k)}(\Gamma^{(k)})^{-1}\Delta^{(k)}\Bigr|_{W_j}$ similarly. Then there exists a non-negative integer $g'$ satisfying
    \begin{align}
        &\sharp(\tau_{1,W_j}\circ \tau_{2,W_j})=km+2-k-2g'-\sharp(\tau_{1,W_j}).
    \end{align}
 
Thus, we obtain
    \begin{align}\label{eqn503}
            &2\cdot \sharp(\mathcal{E}^{(k)}_{j} \Gamma^{(k)} \Delta^{(k)} (\Gamma^{(k)})^{-1} \mathcal{E}^{(k)}_{j} \vee \sigma \Delta^{(k)} \sigma^{-1})\\
            &=\sharp(\mathcal{E}^{(k)}_{j} \Delta^{(k)} \sigma \Delta^{(k)} \sigma^{-1} \mathcal{E}^{(k)}_{j} \Gamma^{(k)} \Delta^{(k)} (\Gamma^{(k)})^{-1} \Delta^{(k)})\\
            &=\sharp(\tau_{1,V_j}\circ \tau_{2,V_j})+\sharp(\tau_{1,W_j}\circ \tau_{2,W_j})\\
            &=2(km+2-k)-2(g+g')-\sharp(\tau_{1,V_j})-\sharp(\tau_{1,W_j})\\
            &=2(km+2-k)-2(g+g')-\sharp(\mathcal{E}^{(k)}_{j} \Delta^{(k)} \sigma \Delta^{(k)} \sigma^{-1} \mathcal{E}^{(k)}_{j})\\
            &=2(km+2-k)-2(g+g')-2\sharp(\sigma),
    \end{align}
which leads us to reach the following conclusion
    \begin{align}
        \notag f_{k,j}(\epsilon,\sigma)=& \sharp(\mathcal{E}^{(k)}_{j} \Gamma^{(k)} \Delta^{(k)} (\Gamma^{(k)})^{-1} \mathcal{E}^{(k)}_{j} \vee \sigma \Delta^{(k)} \sigma^{-1}) + \sharp(\sigma) - k(m+1)\\
        &=2-2k-(g+g')\leq 2-2k.
    \end{align}

\end{proof}

\section{A central limit theorem for partial transposes}\label{sec-CLT}

Let us denote by $\mathbf{d}=(d_{1},\cdots,d_{n})$ and by $\displaystyle \mu(\mathbf{d})= \min_{j\in [n]} d_{j}\geq 2$ in this section. Let $p=p(\mathbf{d})$ and $n=n(\mathbf{d})$ be functions of $\mathbf{d}$. Recall that there exist partial transposes $\left\{W^{\epsilon}\right\}_{\epsilon\in \left\{0,1\right\}^n}\subseteq \bigotimes_{j=1}^{n} M_{d_{j}}(\Comp)$ of the multipartite Wishart matrices $W=W_{d_{1} \cdots d_{n},p}$. In this section, the following product
\begin{equation}
    D(\epsilon, \sigma) =\prod_{j=1}^n d_j^{f_{j}(\epsilon,\sigma)}.
\end{equation}
will play a crucial role.

To establish a central limit theorem, we take a family $\left\{a_{\epsilon}\right\}_{\epsilon }$ consisting of the centered partial transposes
\begin{equation}
   a_{\epsilon} = W^{\epsilon}-c\cdot \text{Id}\in \bigotimes_{j=1}^{n} M_{d_{j}}(\Comp).
\end{equation}
The main result of this section is that, if we take subsets $B_{\mathbf{d}}\subseteq \left\{0,1\right\}^{n({\mathbf{d}})}$ such that $\displaystyle \lim |B_{\bf d}|^m \left ( \frac{1}{\mu(\bf d)}+\left | \frac{p}{d_1d_2\cdots d_n}-c\right | \right )=0$ for all $m\in \n$ and $\lim |B_{\bf d}|=\infty$, then the following random matrices
\begin{equation}
   s_{\mathbf{d}}=\frac{1}{\sqrt{|B_{\mathbf{d}}|}}\sum_{\epsilon\in B_{\mathbf{d}}} a_{\epsilon}
\end{equation}
 converge in moments to the semi-circular element of mean $0$ and variance $c$. To compute the limit of $m$-th moments $\displaystyle \lim(\mathbb{E}\otimes \text{tr})(s_{\mathbf{d}}^m)$, we will focus on the following values
\begin{equation}
    \left ( \mathbb{E}\otimes \text{tr}\right )(a_{x(1)}a_{x(2)}\cdots a_{x(m)})
\end{equation}
where $x:[m]\rightarrow B_{\mathbf{d}}$ is an arbitrary function. Our basic strategy is to recover the arguments in \cite[Section 2.1]{MS17} with detailed analysis of asymptotic bounds. Let us begin with the following lemma.

\begin{lemma}\label{welldef}
    For arbitrary functions $\epsilon:[m]\rightarrow B_{\mathbf{d}}$ and $\epsilon':[m]\rightarrow B_{\mathbf{d}}$ such that $\ker(\epsilon)=\ker(\epsilon')$, we have
    \begin{align}
       &\left |(\mathbb{E} \otimes \text{tr}) (a_{\epsilon(1)} \cdots a_{\epsilon(m)}) - (\mathbb{E} \otimes \text{tr}) ( a_{\epsilon'(1)} \cdots a_{\epsilon'(m)}) \right | \\
       &\leq \frac{2^{m+1}\cdot m! \cdot (1+c)^{m}}{\mu(\mathbf{d})}  \sum_{s=0}^{m} \left ( \frac{p}{d_{1} \cdots d_{n}} \right )^{s}.
    \end{align}
\end{lemma}

\begin{proof}
Let us begin with the following formula
    \begin{align}
        &a_{\epsilon(1)} \cdots a_{\epsilon(m)} \\
        & = (W^{\epsilon(1)}-c\cdot \text{Id}) \cdots (W^{\epsilon(m)}-c\cdot \text{Id}) =  \sum_{E\subseteq [m]} (-c)^{m-|E|}  \prod_{t\in E} W^{\epsilon(t)} ,
    \end{align}
    and write $l=|E|$ for simplicity.
    Consider $\epsilon|_{E}$ as a function from $[l] = [|E|]$ into $B_{\mathbf{d}}$.
    Then, since
   \begin{align}
       (\mathbb{E}\otimes \text{tr})\left ( \prod_{t\in E} W^{\epsilon(t)}\right ) &= \sum_{\tau \in S_{l}} \left( \frac{p}{d_{1} \cdots d_{n}} \right)^{\sharp(\tau)} \prod_{j=1}^{n} d_{j}^{f_{j}(\epsilon|_{E}, \tau)}\\
    & = \sum_{\tau \in S_{l}} \left( \frac{p}{d_{1} \cdots d_{n}} \right)^{\sharp(\tau)} D(\epsilon|_{E}, \tau)
    \end{align}
for each $E \subseteq [m]$ by Theorem \ref{thm302}, we have 
\begin{align}\label{eq600}
    &\left |(\mathbb{E} \otimes \text{tr}) (a_{\epsilon(1)} \cdots a_{\epsilon(m)}) - (\mathbb{E} \otimes \text{tr}) ( a_{\epsilon'(1)} \cdots a_{\epsilon'(m)}) \right |\\
    & = \left | \sum_{E\subseteq [m]} (-c)^{m-|E|}\cdot \sum_{\tau \in S_{l}} \left( \frac{p}{d_{1} \cdots d_{n}} \right)^{\sharp(\tau)} \left [ D(\epsilon|_{E}, \tau) - D(\epsilon'|_{E}, \tau)  \right ] \right | .
\end{align}

Note that the given condition $\ker(\epsilon)=\ker(\epsilon')$ implies 
\begin{equation}
    D(\epsilon|_{E},\tau)=1~\Longleftrightarrow D(\epsilon'|_{E},\tau)=1.
\end{equation}
Indeed, any restriction $\epsilon|_E$ can be considered a function which is of the form
\begin{equation}
    \epsilon|_E(i)=\sum_{z\in B_{\mathbf{d}}} z\cdot \chi_{E_z}(i)
\end{equation}
for all $i\in [l]$ where $E_z=(\epsilon|_E)^{-1}(z)=\left\{i\in [l]: \epsilon|_{E}(i)=z\right\}$. Then the given assumption $D(\epsilon|_{E},\tau)=1$ implies that the partition of $\tau$ is non-crossoing by Lemma \ref{lem400} (2), and $\epsilon|_E$ is constant on each cycle of $\tau$, i.e. each $E_z$ is a union of cycles of $\tau$ by Lemma \ref{lem400} (1). Thus, we can write 
\begin{equation}
\epsilon|_E(i)=\sum_{z\in B_{\mathbf{d}}} z\cdot \chi_{E_z}(i)=\sum_{z\in B_{\mathbf{d}}}\sum_p z\cdot \chi_{\mathfrak{c}_{z,p}}(i)
\end{equation}
where $\displaystyle E_z=\bigcup_{p}\mathfrak{c}_{z,p}$ and $\mathfrak{c}_{z,p}$'s are the disjoint cycles of $\tau$. 

On the other hand, $\epsilon'|_E$ is also written as
\begin{equation}
\epsilon'|_E(i)=\sum_{w\in B_{\mathbf{d}}} w\cdot \chi_{E_w'}(i)
\end{equation}
where $E_w'=(\epsilon'|_E)^{-1}(w)=\left\{i\in [l]: \epsilon'|_{E}(i)=w\right\}$, and the given condition $\text{ker}(\epsilon)=\text{ker}(\epsilon')$ forces $E_w'$ to be equal to one of $E_z$, which is a union of cycles $\mathfrak{c}_{z,p}$ of $\tau$.
Hence, $\epsilon'|_E$ is a linear combination of characteristic functions on cycles of $\tau$.
In other words, $\epsilon'|_E$ is constant on each cycle of $\tau$.
This leads us to conclude that $D(\epsilon'|_{E},\tau)=1$ by Lemma \ref{lem400} since the partition of $\tau$ is non-crossing.

Let us return to \eqref{eq600}. Using the above conclusion
\begin{equation}
D(\epsilon|_{E},\tau)=1\Leftrightarrow D(\epsilon'|_{E},\tau)=1,    
\end{equation}
we obtain
\begin{align}
    &\left |(\mathbb{E} \otimes \text{tr}) (a_{\epsilon(1)} \cdots a_{\epsilon(m)}) - (\mathbb{E} \otimes \text{tr}) ( a_{\epsilon'(1)} \cdots a_{\epsilon'(m)}) \right |\\
    &\leq  \sum_{E\subseteq [m]} c^{m-|E|}\cdot \sum_{\tau \in S_{l}:  D(\epsilon|_{E},\tau)<1 } \left( \frac{p}{d_{1} \cdots d_{n}} \right)^{\sharp(\tau)} \left | {D(\epsilon|_{E},\tau)}-{D(\epsilon'|_{E},\tau)} \right |.
    \end{align}
Furthermore, since
\begin{align}
    &\sum_{\tau \in S_{l}:  D(\epsilon|_{E},\tau)<1 } \left( \frac{p}{d_{1} \cdots d_{n}} \right)^{\sharp(\tau)} \left | {D(\epsilon|_{E},\tau)}-{D(\epsilon'|_{E},\tau)} \right |\\
    \notag&\leq \sum_{\tau \in S_{l}:  D(\epsilon|_{E},\tau)<1 } \left( \frac{p}{d_{1} \cdots d_{n}} \right)^{\sharp(\tau)} {D(\epsilon|_{E},\tau)}\\
    &{\color{white}ttttttttt}+\sum_{\tau \in S_{l}:  D(\epsilon'|_{E},\tau)<1 } \left( \frac{p}{d_{1} \cdots d_{n}} \right)^{\sharp(\tau)} {D(\epsilon'|_{E},\tau)}\\
    & \leq 2 \cdot \sum_{\tau \in S_{l}:  D(\epsilon|_{E},\tau)<1 } \left( \frac{p}{d_{1} \cdots d_{n}} \right)^{\sharp(\tau)} \cdot \frac{1}{\mu(\mathbf{d})} \leq  \frac{2\cdot m!}{\mu(\mathbf{d})}  \sum_{s=0}^{m} \left ( \frac{p}{d_{1} \cdots d_{n}} \right )^{s},
\end{align}
we can conclude that
\begin{align}
    &\left |(\mathbb{E} \otimes \text{tr}) (a_{\epsilon(1)} \cdots a_{\epsilon(m)}) - (\mathbb{E} \otimes \text{tr}) ( a_{\epsilon'(1)} \cdots a_{\epsilon'(m)}) \right |\\
    &\leq \sum_{E\subseteq [m]} c^{m-|E|}\cdot \frac{2\cdot m!}{\mu(\mathbf{d})}  \sum_{s=0}^{m} \left ( \frac{p}{d_{1} \cdots d_{n}} \right )^{s}\\
    &\le \sum_{E\subseteq [m]} (1+c)^{m}\cdot \frac{2\cdot m!}{\mu(\mathbf{d})}  \sum_{s=0}^{m} \left ( \frac{p}{d_{1} \cdots d_{n}} \right )^{s}\\
    &\le \frac{2^{m+1}\cdot m! \cdot (1+c)^{m}}{\mu(\mathbf{d})}  \sum_{s=0}^{m} \left ( \frac{p}{d_{1} \cdots d_{n}} \right )^{s}.
    \end{align}

\end{proof}

The above Lemma \ref{welldef} allows us to rely on $\text{ker}(\epsilon)$ whose structure is categorized into the following four distinct cases:
\begin{itemize}
    \item ({\textbf{Case A}}) $\text{ker}(\epsilon)$ contains a singleton element
    \item ({\textbf{Case B}}) $\text{ker}(\epsilon)$ does not contain a singleton element, and $\epsilon$ is not a pairing.
    \item (\textbf{Case C}) $\text{ker}(\epsilon)$ is a pairing and there exists $i\in [m-1]$ such that $\left\{i,i+1\right\}\in \text{ker}(\epsilon)$
    \item (\textbf{Case D}) $\text{ker}(\epsilon)$ is a pairing, and $\epsilon(i)\neq \epsilon(i+1)$ for all $i\in [m-1]$.
\end{itemize}

Our strategy is to prove Lemma \ref{singleton}, Lemma \ref{double}, Lemma \ref{freeness} to cover (\textbf{Case A}), (\textbf{Case C}), (\textbf{Case D}) respectively, and the following technical Lemma \ref{lem-technical} plays is an important ingredient to establish Lemma \ref{singleton} and Lemma \ref{double}.

\begin{lemma}\label{lem-technical}
\begin{enumerate}
    \item For any $\tau\in S_l$, let us denote by $\tau_1=\tau\circ (l+1)\in S_{l+1}$. Then we have \begin{equation}
\sharp(\mathcal{E} \Gamma \Delta \Gamma^{-1} \mathcal{E} \vee \tau \Delta \tau^{-1}) = \sharp(\mathcal{E}^{'} \Gamma^{'} \Delta^{'} (\Gamma^{'})^{-1} \mathcal{E}^{'} \vee \tau_1 \Delta^{'} \tau_1^{-1}).
\end{equation}
Here, $\mathcal{E}$, $\Delta$, $\Gamma$ are acting on $[\pm l]$, whereas $\mathcal{E}'$, $\Delta'$, $\Gamma'$ are the analogous permutations on $[\pm (l+1)]$.
\item For any $\tau\in S_l$, let us denote by $\tau_2=\tau\circ (l+1,l+2)\in S_{l+2}$. Then we have 
\begin{equation}
\sharp(\mathcal{E} \Gamma \Delta \Gamma^{-1} \mathcal{E} \vee \tau \Delta \tau^{-1})+1 = \sharp(\mathcal{E}^{''} \Gamma^{''} \Delta^{''} (\Gamma^{''})^{-1} \mathcal{E}^{''} \vee \tau_2 \Delta^{''} \tau_2^{-1})
\end{equation}
if $\text{sgn}(\mathcal{E}''(l+1))=\text{sgn}(\mathcal{E}''(l+2))$. Here, $\mathcal{E}'',\Delta'',\Gamma''$ are the analogous permutations on $[\pm (l+2)]$.
\end{enumerate}

\end{lemma}

\begin{proof}

(1) First of all, it is straightforward to check the following facts:
\begin{enumerate}
    \item [(A)] $\tau_1 \Delta' \tau_1^{-1}=\tau \Delta \tau^{-1} \circ (-(l+1),l+1)$,
    \item [(B)] $(-\mathcal{E}(l),\mathcal{E}(1))$ is one of the disjoint cycles in $\mathcal{E} \Gamma \Delta \Gamma^{-1} \mathcal{E}$,
    \item [(C)] $\mathcal{E} \Gamma \Delta \Gamma^{-1} \mathcal{E}=\mathcal{E}^{'} \Gamma^{'} \Delta^{'} (\Gamma^{'})^{-1} \mathcal{E}^{'}$ on $[\pm l]\setminus \left\{-\mathcal{E}(l),\mathcal{E}(1)\right\}$,
    \item [(D)] $(-\mathcal{E}'(l+1),\mathcal{E}(1))$ and $(\mathcal{E}'(l+1),-\mathcal{E}(l))$ are disjoint cycles of $\mathcal{E}' \Gamma' \Delta' (\Gamma')^{-1} \mathcal{E}'$. In particular, we have 
    \begin{equation}
    \mathcal{E} \Gamma \Delta \Gamma^{-1} \mathcal{E}= \mathcal{E}' \Gamma' \Delta' (\Gamma')^{-1} \mathcal{E}'\circ \tau_1 \Delta' \tau_1^{-1} \circ \mathcal{E}' \Gamma' \Delta' (\Gamma')^{-1} \mathcal{E}'.
    \end{equation}
    on $\left\{-\mathcal{E}(l),\mathcal{E}(1)\right\}$.
\end{enumerate}

Let us suppose that $\left\{B_1,B_2,\cdots,B_N\right\}$ is the disjoint decomposition of blocks of $\mathcal{E} \Gamma \Delta \Gamma^{-1} \mathcal{E} \vee \tau \Delta \tau^{-1}$, and we may assume that
\begin{equation}
    B_1=\left\{\mathcal{E}(1),-\mathcal{E}(l)\right\}\cup T \subseteq [\pm l]
\end{equation}
since (B) implies $-\mathcal{E}(l)\sim \mathcal{E}(1)$ by $\mathcal{E} \Gamma \Delta \Gamma^{-1} \mathcal{E}$. Here, we may assume that $T$ is disjoint from $\left\{\mathcal{E}(1),-\mathcal{E}(l)\right\}$.

On the other hand, we now claim that 
\begin{equation}
    \{B_1\cup \{ \pm \mathcal{E}^{'}(l+1) \}\}\cup \{B_2,\cdots,B_N\}
\end{equation} 
is the disjoint decomposition of blocks of $\mathcal{E}^{'} \Gamma^{'} \Delta^{'} (\Gamma^{'})^{-1} \mathcal{E}^{'} \vee \tau_1 \Delta^{'} \tau_1^{-1}$. Indeed, (A) and (C) explain why $B_2,\cdots,B_N$ are the disjoint blocks, so the only remaining part is to prove that 
\begin{equation}
    B_1\cup \{ \pm \mathcal{E}^{'}(l+1) \}=\left\{\mathcal{E}(1),-\mathcal{E}(l)\right\}\cup \{ \pm \mathcal{E}^{'}(l+1) \} \cup T
\end{equation} 
is a disjoint block of $\mathcal{E}^{'} \Gamma^{'} \Delta^{'} (\Gamma^{'})^{-1} \mathcal{E}^{'} \vee \tau_1 \Delta^{'} \tau_1^{-1}$. Firstly, let us prove that any elements $x,x'$ in $T$ are connected by $\mathcal{E}^{'} \Gamma^{'} \Delta^{'} (\Gamma^{'})^{-1} \mathcal{E}^{'}$ and $\tau_1 \Delta^{'} \tau_1^{-1}$. Our assumption provides a sequence $(x_i)_{i=0}^t$ such that $x_0=x\in T$, $x_t=x'\in T$ and $x_{i}=(\mathcal{E} \Gamma \Delta \Gamma^{-1} \mathcal{E})(x_{i-1})$ or $x_{i}=(\tau \Delta \tau^{-1})(x_{i-1})$ for each $i\in [t]$. If $(x_i)_{i=0}^t\subseteq T$, then all the actions of $\mathcal{E} \Gamma \Delta \Gamma^{-1} \mathcal{E}$ and $\tau \Delta \tau^{-1}$ coincide with the actions of $\mathcal{E}' \Gamma' \Delta' (\Gamma')^{-1} \mathcal{E}'$ and $\tau_1 \Delta' \tau_1^{-1}$ by (A) and (C), so the conclusion follows immediately. Now, if we suppose that $x_{i}=(\tau \Delta \tau^{-1})(x_{i-1})\in \left\{\mathcal{E}(1),-\mathcal{E}(l)\right\}$ at the $i$-th step, then we may assume that the next two elements are given by 
\begin{equation}
    \begin{array}{ll}
&x_{i+1}=(\mathcal{E} \Gamma \Delta \Gamma^{-1} \mathcal{E})(x_{i})\in \left\{\mathcal{E}(1),-\mathcal{E}(l)\right\},\\
&x_{i+2}=(\tau \Delta \tau^{-1})(x_{i+1})\in T.
\end{array}
\end{equation}
Furthermore, the action $\mathcal{E} \Gamma \Delta \Gamma^{-1} \mathcal{E}$ at the $(i+1)$-th step can be replaced by 
\begin{equation}
    \mathcal{E}' \Gamma' \Delta' (\Gamma')^{-1} \mathcal{E}'\circ \tau_1 \Delta' \tau_1^{-1} \circ \mathcal{E}' \Gamma' \Delta' (\Gamma')^{-1} \mathcal{E}'
\end{equation} 
as noted in (D), and the action $\tau \Delta \tau^{-1}$ coincides with $\tau_1 \Delta' \tau_1^{-1}$ by (A). Thus, we can conclude that $x_0=x$ and $x_t=x'$ are connected by the pairings $\mathcal{E}' \Gamma' \Delta' (\Gamma')^{-1} \mathcal{E}'$ and $\tau_1 \Delta' \tau_1^{-1}$. For example, if $x_i=-\mathcal{E}(l)$, then the original sequence $\cdots, x_{i-1}, x_i, x_{i+1}, x_{i+2},\cdots$ corresponds to the blue-green-blue paths, and the green path from $x_i$ to $x_{i+1}$ is replaced by the three red paths in the following figure:
\begin{figure}[htb!]
    \centering
    \includegraphics[scale=0.7]{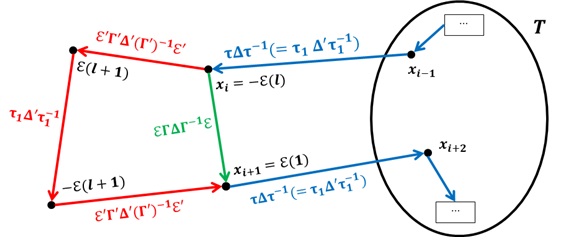}
\end{figure}

Furthermore, (A) and (D) tell us that all elements of $\left\{\mathcal{E}(1),-\mathcal{E}(l)\right\}\cup \left\{ \pm \mathcal{E}'(l+1) \right\}$ are connected by $\mathcal{E}^{'} \Gamma^{'} \Delta^{'} (\Gamma^{'})^{-1} \mathcal{E}^{'}$ and $\tau_1 \Delta^{'} \tau_1^{-1}$. Lastly, if we assume there is no element of $T$ connected to $\left\{\mathcal{E}(1),-\mathcal{E}(l)\right\}\cup \left\{ \pm \mathcal{E}'(l+1) \right\}$, then it implies that $T$ is one of the disjoint blocks of $\mathcal{E} \Gamma \Delta \Gamma^{-1} \mathcal{E} \vee \tau \Delta \tau^{-1}$, which contradicts to the fact that $T$ is a strict subset of $B_1$.

(2) In this case, it is straightforward to check the following facts
\begin{enumerate}
    \item [(A)] $\tau_2 \Delta'' \tau_2^{-1} = \tau \Delta \tau^{-1} \circ (-(l+2),l+1)\circ (-(l+1),l+2)$,
    \item [(B)] $(-\mathcal{E}(l),\mathcal{E}(1))$ is one of the disjoint cycles in $\mathcal{E} \Gamma \Delta \Gamma^{-1} \mathcal{E}$,
    \item [(C)] $\mathcal{E} \Gamma \Delta \Gamma^{-1} \mathcal{E}=\mathcal{E}^{''} \Gamma^{''} \Delta^{''} (\Gamma^{''})^{-1} \mathcal{E}^{''}$ on $[\pm l]\setminus \left\{-\mathcal{E}(l),\mathcal{E}(1)\right\}$,
    \item [(D)] $(\mathcal{E}(1),-\mathcal{E}''(l+2))$, $(\mathcal{E}''(l+1),-\mathcal{E}(l))$ are cycles of $\mathcal{E}^{''} \Gamma^{''} \Delta^{''} (\Gamma^{''})^{-1} \mathcal{E}^{''}$, and $(-\mathcal{E}''(l+2),\mathcal{E}''(l+1))$ is a cycle of $\tau_2\Delta''\tau_2^{-1}$. In particular, we have
    \begin{equation}
        \mathcal{E} \Gamma \Delta \Gamma^{-1} \mathcal{E}=\mathcal{E}^{''} \Gamma^{''} \Delta^{''} (\Gamma^{''})^{-1} \mathcal{E}^{''} \circ \tau_2 \Delta^{''} \tau_2^{-1}\circ \mathcal{E}^{''} \Gamma^{''} \Delta^{''} (\Gamma^{''})^{-1} \mathcal{E}^{''}
    \end{equation}
    on $\left\{\mathcal{E}(1),-\mathcal{E}(l)\right\}$.
\end{enumerate}

As in the proof of (1), let us suppose that $\left\{B_1,B_2,\cdots,B_N\right\}$ is the disjoint block decomposition of $\mathcal{E} \Gamma \Delta \Gamma^{-1} \mathcal{E} \vee \tau \Delta \tau^{-1}$, and we may assume that
\begin{equation}
    B_1=\left\{\mathcal{E}(1),-\mathcal{E}(l)\right\}\cup T \subseteq [\pm l]
\end{equation}
and $T$ is disjoint from $\left\{\mathcal{E}(1),-\mathcal{E}(l)\right\}$ since (B) implies $-\mathcal{E}(l)\sim \mathcal{E}(1)$ by $\mathcal{E} \Gamma \Delta \Gamma^{-1} \mathcal{E}$.

From now on, we will claim that there exist precisely $N+1$ disjoint blocks of $\mathcal{E}^{''} \Gamma^{''} \Delta^{''} (\Gamma^{''})^{-1} \mathcal{E}^{''} \vee \tau_2 \Delta^{''} \tau_2^{-1}$. Indeed, (A) and (C) imply that $B_2,B_3,\cdots,B_N$ are $N-1$ disjoint blocks and it is immediate to check that $\left (\mathcal{E}''(l+2),-\mathcal{E}''(l+1) \right )$ is a cycle of both $\mathcal{E}^{''} \Gamma^{''} \Delta^{''} (\Gamma^{''})^{-1} \mathcal{E}^{''}$ and $\tau_2 \Delta^{''} \tau_2^{-1}$. Thus, the only remaining part is to prove that all elements in
\begin{align}
    &B_1\cup \left\{ \mathcal{E}''(l+1),-\mathcal{E}''(l+2) \right\}\\
    &=T\cup \left\{\mathcal{E}(1),-\mathcal{E}(l)\right\} \cup \left\{ \mathcal{E}''(l+1),-\mathcal{E}''(l+2) \right\}
\end{align}
are connected by $\mathcal{E}^{''} \Gamma^{''} \Delta^{''} (\Gamma^{''})^{-1} \mathcal{E}^{''}$ and $\tau_2 \Delta^{''} \tau_2^{-1}$. Firstly, all elements in $T$ are connected by (A), (C), (D), and all elements in $\left\{\mathcal{E}(1),-\mathcal{E}''(l+2)\right\}\cup \left\{\mathcal{E}''(l+1),-\mathcal{E}(l)\right\}$ are also connected by $\mathcal{E}^{''} \Gamma^{''} \Delta^{''} (\Gamma^{''})^{-1} \mathcal{E}^{''}$ and $\tau_2 \Delta^{''} \tau_2^{-1}$ thanks to (D) as in the proof of (1). Then, if we assume that there is no element of $T$ connected to 
\begin{equation}
    \left\{\mathcal{E}(1),-\mathcal{E}''(l+2)\right\}\cup \left\{\mathcal{E}''(l+1),-\mathcal{E}(l)\right\},
\end{equation} 
then $T$ should be one of the disjoint blocks of $\mathcal{E} \Gamma \Delta \Gamma^{-1} \mathcal{E} \vee \tau \Delta \tau^{-1}$. This contradicts to the fact that $T$ is a strict subset of $B_1$.

\end{proof}

Now, let us present an estimate of $(\mathbb{E}\otimes\text{tr})(a_{\epsilon(1)}\cdots a_{\epsilon(m)})$ for ({\bf Case A}).

\begin{lemma}\label{singleton}
    Let $\epsilon:[m]\rightarrow B_{\mathbf{d}}$ be a function and suppose that   $\text{ker}(\epsilon)$ contains a singleton set. Then we have
    \begin{align}
    \notag&\left | \left ( \mathbb{E}\otimes \text{tr}\right )(a_{\epsilon(1)}a_{\epsilon(2)}\cdots a_{\epsilon(m)}) \right |\\
    &\leq 2^{m}(1+c)^{m}m!\left (\frac{1}{\mu(\bf d)}+\left | \frac{p}{d_1d_2\cdots d_n}-c\right | \right )\sum_{s=0}^m \left ( \frac{p}{d_{1} \cdots d_{n}} \right )^s.
    \end{align}
\end{lemma}

\begin{proof}

We may assume $\epsilon(i)\neq \epsilon(m)$ for all $i\in [m-1]$ thanks to the given assumption and the traciality of $\mathbb{E}\otimes \text{tr}$.  Let us begin with the following formula
    \begin{align}
        &a_{\epsilon(1)} \cdots a_{\epsilon(m)} = (W^{\epsilon(1)}-c\cdot \text{Id}) \cdots (W^{\epsilon(m)}-c\cdot \text{Id})\\
        & = \left [ \sum_{E\subseteq[m-1]} (-c)^{(m-1)-|E|}  \prod_{t\in E} W^{\epsilon(t)} \right ] (W^{\epsilon(m)}-c\cdot \text{Id}),
    \end{align}
    and write $l=|E|$ for simplicity. Note that 
    \begin{align}
        (\mathbb{E}\otimes\text{tr}) \left ( \left [\prod_{t\in E} W^{\epsilon(t)}\right ] \cdot W^{\epsilon(m)}\right )&= \sum_{\sigma \in S_{l+1}} \left( \frac{p}{d_{1} \cdots d_{n}} \right)^{\sharp(\sigma)} {D(\epsilon|_{E \cup \{ m \}},\sigma)}\\
        (\mathbb{E}\otimes \text{tr})\left ( \prod_{t\in E} W^{\epsilon(t)}\right ) &= \sum_{\tau \in S_{l}} \left( \frac{p}{d_{1} \cdots d_{n}} \right)^{\sharp(\tau)} {D(\epsilon|_{E},\tau)}
    \end{align}
for each $E \subseteq [m-1]$ by Theorem \ref{thm302}. Now, let us understand $\sigma\in S_{l+1}$ as a permutation acting on $[l+1]\cong E\cup \left\{m\right\}$. Then the image of the map $\tau\mapsto \tau_1=\tau\circ (l+1)$ consists of the permutations $\sigma\in S_{l+1}$ satisfying $\sigma(l+1)=l+1$. Furthermore, Lemma \ref{lem-technical} provides the following identity
\begin{align}
    &f_{j}(\epsilon|_{E \cup \{ m \}},\tau_1)=\sharp(\mathcal{E}_{j}^{'} \Gamma^{'} \Delta^{'} (\Gamma^{'})^{-1} \mathcal{E}_{j}^{'} \vee \tau_1 \Delta^{'} \tau_1^{-1}) + \sharp(\tau_1) - (l+2)\\
    &=\sharp(\mathcal{E}_{j} \Gamma \Delta \Gamma^{-1} \mathcal{E}_{j} \vee \tau \Delta \tau^{-1}) + (\sharp(\tau)+1) - (l+2)\\
    &=\sharp(\mathcal{E}_{j} \Gamma \Delta \Gamma^{-1} \mathcal{E}_{j} \vee \tau \Delta \tau^{-1}) + \sharp(\tau) - (l+1)=f_{j}(\epsilon|_{E},\tau),
\end{align}
and we obtain
\begin{align}
    &(\mathbb{E}\otimes \text{tr}) \left (\left [\prod_{t\in E} W^{\epsilon(t)} \right ] (W^{\epsilon(m)}-c\cdot \text{Id}) \right )\\
    &=\underset{\sigma(l+1)\neq l+1}{\sum_{\sigma \in S_{l+1}:}} \left( \frac{p}{d_{1} \cdots d_{n}} \right)^{\sharp(\sigma)} {D(\epsilon|_{E \cup \{ m \}},\sigma)}\\
    &{\color{white}tttttt}+ \sum_{\tau \in S_{l}}\left( \frac{p}{d_{1} \cdots d_{n}}-c \right) \left( \frac{p}{d_{1} \cdots d_{n}} \right)^{\sharp(\tau)} {D(\epsilon|_{E},\tau)}.
\end{align}
In particular, for $\sigma\in S_{l+1}$ with $\sigma(l+1)\neq l+1$, the given assumption  $\left\{m\right\}\in \text{ker}(\epsilon)$ implies that $\epsilon(t)\neq \epsilon(m)$ for any $t\in E\subseteq [m-1]$, i.e. there exists $j\in [n]$ such that $\epsilon(t)_{j} \ne \epsilon(m)_{j}$.
This means that $[\epsilon|_{E \cup \{ m \}}(\cdot)]_j$ is not constant on the cycle containing $l+1$, so we should have $f_{j}(\epsilon|_{E \cup \{ m \}},\sigma) \le -1$ for some $j$ and ${D(\epsilon|_{E \cup \{ m \}},\sigma)}\leq \mu(\mathbf{d})^{-1}$ by Lemma \ref{lem400} (1).
Then, combining all the discussions above with the standard triangle inequality, we obtain
    \begin{align}
        &\left |( \mathbb{E}\otimes\text{tr} ) \left (\left [\prod_{t\in E} W^{\epsilon(t)} \right ] (W^{\epsilon(m)}-c\cdot \text{Id}) \right ) \right |\\
        &\le m!\sum_{s=0}^{m} \left (\frac{p}{d_{1} \cdots d_{n}} \right )^s\cdot \frac{1}{\mu(\mathbf{d})} +m!\left |\frac{p}{d_{1} \cdots d_{n}}-c\right | \sum_{s=0}^m \left (\frac{p}{d_{1} \cdots d_{n}} \right )^s\\
        &\leq m! \left (\frac{1}{\mu(\mathbf{d})}+\left | \frac{p}{d_{1} \cdots d_{n}}-c \right | \right )\sum_{s=0}^m \left ( \frac{p}{d_{1} \cdots d_{n}} \right )^s.
    \end{align}

 Hence, we reach the following conclusion
    \begin{align}
        &\left | ( \mathbb{E}\otimes\text{tr})(a_{\epsilon(1)} \cdots a_{\epsilon(m)}) \right | \\
        &\le \sum_{E\subseteq [m-1]} c^{(m-1)-|E|} m! \left (\frac{1}{\mu(\mathbf{d})}+\left | \frac{p}{d_{1} \cdots d_{n}}-c \right | \right )\sum_{s=0}^m \left ( \frac{p}{d_{1} \cdots d_{n}} \right )^s\\
        &\leq 2^{m}(1+c)^{m} m! \left (\frac{1}{\mu(\mathbf{d})}+\left | \frac{p}{d_{1} \cdots d_{n}}-c \right | \right )\sum_{s=0}^m \left ( \frac{p}{d_{1} \cdots d_{n}} \right )^s
    \end{align}
 
\end{proof}

Now, the following Lemma provides an estimate of $(\mathbb{E}\otimes\text{tr})(a_{\epsilon(1)}\cdots a_{\epsilon(m)})$ for ({\bf Case C}).

\begin{lemma}\label{double}
Let $\epsilon:[m]\rightarrow B_{\mathbf{d}}$ be a function and suppose that there exists $i\in [m-1]$ such that $\left\{i,i+1\right\}\in \text{ker}(\epsilon)$. Then we have
    \begin{align}
        &\left| (\mathbb{E}\otimes \text{tr})\left (\prod_{j\in [m]} a_{\epsilon(j)}\right )-c\cdot (\mathbb{E}\otimes \text{tr})\left (\prod_{j\in [m]\setminus\left\{i,i+1\right\}} a_{\epsilon(j)}\right ) \right |\\
        & \le 2^{m}(1+c)^m (2+c) \left (\frac{1}{\mu(\mathbf{d})}+\left | \frac{p}{d_{1} \cdots d_{n}}-c \right | \right ) m! \sum_{s=0}^m \left ( \frac{p}{d_{1} \cdots d_{n}} \right )^s.
    \end{align}

\end{lemma}

\begin{proof}
    As in the proof of Lemma \ref{singleton}, we may assume $\left\{m-1,m\right\}\in \text{ker}(\epsilon)$ using the traciality of $\mathbb{E}\otimes \text{tr}$ and we have following identity
    \begin{align}
        &a_{\epsilon(1)} \cdots a_{\epsilon(m)} - c a_{\epsilon(1)} \cdots a_{\epsilon(m-2)}\\
        &=\sum_{E\subseteq[m-2]} (-c)^{m-(2+|E|)} \left[ \prod_{t\in E} W^{\epsilon(t)} \right] (W^{\epsilon(m-1)}-c\cdot \text{Id})(W^{\epsilon(m)}-c\cdot \text{Id})\\
        &{\color{white}ttttttt}-c \sum_{E\subseteq[m-2]} (-c)^{m-(2+|E|)} \left[ \prod_{t\in E} W^{\epsilon(t)} \right]
    \end{align}
    Let us write $l=|E|$. Then similar arguments from the proof of Lemma \ref{singleton} give us the following two identities:
\begin{align}
     &(\mathbb{E}\otimes\text{tr})\left (\left[ \prod_{t\in E} W^{\epsilon(t)} \right]W^{\epsilon(m-1)}W^{\epsilon(m)} -c\cdot  \left[ \prod_{t\in E} W^{\epsilon(t)} \right] W^{\epsilon(m-1)}\right )\\
    \notag &=\sum_{\substack{\rho \in S_{l+2}:\\\rho(l+2)\neq l+2 }} \left( \frac{p}{d_{1} \cdots d_{n}} \right)^{\sharp(\rho)} {D(\epsilon|_{E \cup \{ m-1, m \}},\rho)}\\
    \label{eq601}&{\color{white}tttt}+\sum_{\sigma \in S_{l+1}}\left( \frac{p}{d_{1} \cdots d_{n}}-c \right) \left( \frac{p}{d_{1} \cdots d_{n}} \right)^{\sharp(\sigma)} {D(\epsilon|_{E \cup \{ m-1 \}},\sigma)},\\
     &(\mathbb{E}\otimes\text{tr})\left (\left[ \prod_{t\in E} W^{\epsilon(t)} \right]W^{\epsilon(m)} -c\cdot  \left[ \prod_{t\in E} W^{\epsilon(t)} \right]\right )\\
    \notag &=\sum_{\substack{\tau \in S_{l+1}:\\ \tau(l+1)\neq l+1 }} \left( \frac{p}{d_{1} \cdots d_{n}} \right)^{\sharp(\tau)} {D(\epsilon|_{E \cup \{ m \}},\tau)}\\
    \label{eq602}&{\color{white}tttttt}+\sum_{v \in S_{l}}\left( \frac{p}{d_{1} \cdots d_{n}}-c \right) \left( \frac{p}{d_{1} \cdots d_{n}} \right)^{\sharp(v)} {D(\epsilon|_{E},v)}.
\end{align}
Here, both the second sums of \eqref{eq601} and  \eqref{eq602} are dominated by 
\begin{equation}
    \left |\frac{p}{d_{1} \cdots d_{n}}-c\right | \cdot m! \sum_{s=0}^{m}\left ( \frac{p}{d_{1} \cdots d_{n}} \right )^{s} ,
\end{equation}
and the first sum of \eqref{eq602} is dominated by $\displaystyle \frac{m!}{\mu(\mathbf{d})} \sum_{s=0}^{m}\left ( \frac{p}{d_{1} \cdots d_{n}} \right )^s $ as in the proof of Lemma \ref{singleton}.
Thus, it is straightforward to check that\footnotesize
\begin{align}
    &\left | (\mathbb{E}\otimes \text{tr})\left ( \left[ \prod_{t\in E} W^{\epsilon(t)} \right] (W^{\epsilon(m-1)}-c\cdot \text{Id})(W^{\epsilon(m)}-c\cdot \text{Id}) - c\cdot  \prod_{t\in E} W^{\epsilon(t)} \right ) \right |\\
    \notag &\leq \left | \sum_{\substack{\rho \in S_{l+2}:\\\rho(l+2)\neq l+2 }} \left( \frac{p}{d_{1} \cdots d_{n}} \right)^{\sharp(\rho)} {D(\epsilon|_{E \cup \{ m-1, m \}},\rho)} -c\cdot \sum_{v \in S_{l} } \left( \frac{p}{d_{1} \cdots d_{n}} \right)^{\sharp(v)} {D(\epsilon|_{E},v)} \right |\\
    \label{eq603}&{\color{white}tttttt}+ (1+c) \left ( \frac{1}{\mu(\mathbf{d})}+\left | \frac{p}{d_{1} \cdots d_{n}}-c \right | \right ) m!\sum_{s=0}^{m}\left ( \frac{p}{d_{1} \cdots d_{n}} \right )^s
\end{align}

\normalsize
Recall that the image of a function $v\in S_l\mapsto v_2=v\circ (l+1,l+2)\in S_{l+2}$ consists of the permutations whose one of the disjoint cycles is $(l+1,l+2)$ and that 
\begin{align}
    &f_{j}(\epsilon|_{E \cup \{ m-1, m \}},v_2)\\
    &=\sharp(\mathcal{E}_{j}^{''} \Gamma^{''} \Delta^{''} (\Gamma^{''})^{-1} \mathcal{E}_{j}^{''} \vee v_2 \Delta^{''} v_2^{-1}) + \sharp(v_2) - (l+3)\\
    &=\left (\sharp(\mathcal{E}_{j} \Gamma \Delta \Gamma^{-1} \mathcal{E}_{j} \vee v \Delta v^{-1})+1 \right ) + (\sharp(v)+1) - (l+3)\\
    &=\sharp(\mathcal{E}_{j} \Gamma \Delta \Gamma^{-1} \mathcal{E}_{j} \vee v \Delta v^{-1}) + \sharp(v) - (l+1)=f_{j}(\epsilon|_{E},v)
\end{align}
by Lemma \ref{lem-technical} (2).
Let us write $c\notin \rho$ if $c$ is not a disjoint cycle of $\rho\in S_{l+2}$.
Then we have\footnotesize
\begin{align}
    &\left | \sum_{\substack{\rho \in S_{l+2}:\\\rho(l+2)\neq l+2 }} \left( \frac{p}{d_{1} \cdots d_{n}} \right)^{\sharp(\rho)} {D(\epsilon|_{E \cup \{ m-1, m \}},\rho)} -c\cdot \sum_{v \in S_{l} } \left( \frac{p}{d_{1} \cdots d_{n}} \right)^{\sharp(v)} {D(\epsilon|_{E},v)} \right |\\
    \notag& \leq \left |  \sum_{\substack{\rho \in S_{l+2}:\\\rho(l+2)\neq l+2 \\ (l+1,l+2)\notin \rho }} \left( \frac{p}{d_{1} \cdots d_{n}} \right)^{\sharp(\rho)} {D(\epsilon|_{E \cup \{ m-1, m \}},\rho)}\right |\\
    &{\color{white}sssssssssttttttttt}+\left | \sum_{v \in S_{l} } \left (\frac{p}{d_{1} \cdots d_{n}}-c \right ) \left( \frac{p}{d_{1} \cdots d_{n}} \right)^{\sharp(v)} {D(\epsilon|_{E},v)}  \right |.
\end{align}

\normalsize
From the conditions $(l+2)\notin \rho$ and $(l+1,l+2)\notin \rho$, there exists $b_0\in [l]$ such that $\rho(b_0)=l+1$ or $\rho(b_0)=l+2$.
Note that
\begin{equation}
    \epsilon|_{E \cup \{ m-1, m \}}(b_0)\neq \epsilon|_{E \cup \{ m-1, m \}}(l+1)=\epsilon|_{E \cup \{ m-1, m \}}(l+2)
\end{equation}
from the given assumption, so there exists $j\in [n]$ such that 
\begin{equation}
    [\epsilon|_{E \cup \{ m-1, m \}}(b_0)]_{j}\neq [\epsilon|_{E \cup \{ m-1, m \}}(l+1)]_{j}=[\epsilon|_{E \cup \{ m-1, m \}}(l+2)]_{j}.
\end{equation} 
This means that $[\epsilon|_{E \cup \{ m-1, m \}}(\cdot)]_j$ is not constant on the cycle containing $b_0$ in $\rho$, implying ${D(\epsilon|_{E \cup \{ m-1, m \}},\rho)}\leq \mu(\mathbf{d})^{-1}$ for such $\rho$ by Lemma \ref{lem400} (1).
Thus, we obtain\footnotesize
\begin{align}
    &\left | \sum_{\substack{\rho \in S_{l+2}:\\\rho(l+2)\neq l+2 }} \left( \frac{p}{d_{1} \cdots d_{n}} \right)^{\sharp(\rho)} {D(\epsilon|_{E \cup \{ m-1, m \}},\rho)} -c\cdot \sum_{v \in S_{l} } \left( \frac{p}{d_{1} \cdots d_{n}} \right)^{\sharp(v)} {D(\epsilon|_{E},v)} \right |\\
    \label{eq604}&\leq \left (\frac{1}{\mu(\mathbf{d})}+\left | \frac{p}{d_{1} \cdots d_{n}}-c \right | \right )m!\sum_{s=0}^m \left ( \frac{p}{d_{1} \cdots d_{n}} \right )^s.
\end{align}

\normalsize
Finally, combining \eqref{eq603} and \eqref{eq604}, we can conclude that
\begin{align}
    & \left | (\mathbb{E}\otimes \text{tr})\left ( a_{\epsilon(1)} \cdots a_{\epsilon(m)}\right ) - c\cdot (\mathbb{E}\otimes \text{tr}) \left ( a_{\epsilon(1)} \cdots a_{\epsilon(m-2)} \right ) \right | \\
    &\leq \sum_{E\subseteq [m-2]} c^{m-(2+|E|)} (2+c) \left (\frac{1}{\mu(\mathbf{d})}+\left | \frac{p}{d_{1} \cdots d_{n}}-c \right | \right )m!\sum_{s=0}^m \left ( \frac{p}{d_{1} \cdots d_{n}} \right )^s\\
    &\leq 2^{m}(1+c)^m (2+c)\left (\frac{1}{\mu(\mathbf{d})}+\left | \frac{p}{d_{1} \cdots d_{n}}-c \right | \right )m!\sum_{s=0}^m \left ( \frac{p}{d_{1} \cdots d_{n}} \right )^s.
\end{align}

\end{proof}

As of the last ingredient to reach the main conclusion, let us present an estimate of $(\mathbb{E}\otimes\text{tr})(a_{\epsilon(1)}\cdots a_{\epsilon(m)})$ for ({\bf Case D}) in the following lemma.

\begin{lemma}\label{freeness}
Let $\epsilon:[m]\rightarrow B_{\mathbf{d}}$ be a function and suppose that $\epsilon(i)\neq \epsilon(i+1)$ for all $i\in [m-1]$. Then we have
\begin{align}
    &\left | (\mathbb{E} \otimes \text{tr}) (a_{\epsilon(1)} \cdots a_{\epsilon(m)}) \right |\\
    &\leq 2^{m} m! \left (\frac{1}{\mu(\mathbf{d})}+ \left | \frac{p}{d_{1} \cdots d_{n}}-c\right | \right ) (m+1)! \left ( \frac{p}{d_{1} \cdots d_{n}}+c+1\right )^{3m}.
\end{align}

\end{lemma}

\begin{proof}
    Recall that
    \begin{align}
        (\mathbb{E} \otimes \text{tr}) (a_{\epsilon(1)} \cdots a_{\epsilon(m)})&=\sum_{E \subseteq [m]} (-c)^{m-|E|} ( \mathbb{E} \otimes \text{tr}) \left (\prod_{t\in E}W^{\epsilon(t)} \right )\\
        =&\sum_{E \subseteq [m]} (-c)^{m-|E|} \sum_{\sigma\in S_{l}} \left ( \frac{p}{d_{1} \cdots d_{n}}\right )^{\sharp(\sigma)} {D(\epsilon|_{E},\sigma)}
    \end{align}
where $l=|E|$, and 
\begin{align}
    0=&\lim_{d_{1}', \cdots, d_{n}' \rightarrow \infty}(\mathbb{E} \otimes \text{tr}) (a_{\epsilon(1)} \cdots a_{\epsilon(m)})\\
    =& \sum_{E\subseteq [m]} (-c)^{m-|E|} \sum_{\sigma\in S_l} c^{\sharp(\sigma)}\lim_{d_{1}', \cdots, d_{n}' \rightarrow \infty}\prod_{j=1}^{n}(d_{j}')^{f_{j}(\epsilon|_{E},\sigma)}
\end{align}
by the asymptotic freeness of $\left\{W^{\sigma}\right\}_{\sigma\in \left\{0,1\right\}^n}$ (Theorem \ref{thm500}). Thus, \footnotesize
\begin{align}
    &\left | (\mathbb{E} \otimes \text{tr}) (a_{\epsilon(1)} \cdots a_{\epsilon(m)}) \right |=\left | (\mathbb{E} \otimes \text{tr}) (a_{\epsilon(1)} \cdots a_{\epsilon(m)})-0 \right |\\
    &\leq \sum_{E \subseteq [m]} c^{m-l} \sum_{\sigma\in S_{l}} \left | \left ( \frac{p}{d_{1} \cdots d_{n}}\right )^{\sharp(\sigma)} {D(\epsilon|_{E},\sigma)} - c^{\sharp(\sigma)}\lim_{d_{1}', \cdots d_{n}' \rightarrow \infty}\prod_{j=1}^{n}(d_{j}')^{f_{j}(\epsilon|_{E},\sigma)} \right |
\end{align}
\normalsize and the standard triangle inequality tells us
\begin{align}
    &\left | \left ( \frac{p}{d_{1} \cdots d_{n}}\right )^{\sharp(\sigma)} {D(\epsilon|_{E},\sigma)} - c^{\sharp(\sigma)}\lim_{d_{1}', \cdots d_{n}' \rightarrow \infty}\prod_{j=1}^{n}(d_{j}')^{f_{j}(\epsilon|_{E},\sigma)} \right |\\
   & \leq \left | \left ( \frac{p}{d_{1} \cdots d_{n}}\right )^{\sharp(\sigma)} - c^{\sharp(\sigma)}\right | {D(\epsilon|_{E},\sigma)}\\
    &{\color{white}tttttt}+ c^{\sharp(\sigma)} \left |  {D(\epsilon|_{E},\sigma)} - \lim_{d_{1}', \cdots d_{n}' \rightarrow \infty}\prod_{j=1}^{n}(d_{j}')^{f_{j}(\epsilon|_{E},\sigma)} \right |.
\end{align}

Furthermore, the binomial theorem implies
\begin{align}
    \left | \left ( \frac{p}{d_{1} \cdots d_{n}}\right )^{\sharp(\sigma)} - c^{\sharp(\sigma)}\right |\leq \left | \frac{p}{d_{1} \cdots d_{n}}-c\right | (m+1)! \left ( \frac{p}{d_{1} \cdots d_{n}}+c+1\right)^{2m}
\end{align}
and it is immediate to see that
\begin{align}
    \left |  {D(\epsilon|_{E},\sigma)} - \lim_{d_{1}', \cdots d_{n}' \rightarrow \infty}\prod_{j=1}^{n}(d_{j}')^{f_{j}(\epsilon|_{E},\sigma)} \right |\leq \frac{1}{\mu(\mathbf{d})}
\end{align}
for all permutations $\sigma\in S_l$. Thus we can conclude that
    \begin{align}
        &\left | \left ( \frac{p}{d_{1} \cdots d_{n}}\right )^{\sharp(\sigma)} {D(\epsilon|_{E},\sigma)} - c^{\sharp(\sigma)}\lim_{d_{1}', \cdots d_{n}' \rightarrow \infty}\prod_{j=1}^{n}(d_{j}')^{f_{j}(\epsilon|_{E},\sigma)} \right |\\
        \notag &\leq  c^{\sharp(\sigma)}\left |  {D(\epsilon|_{E},\sigma)} - \lim_{d_{1}', \cdots d_{n}' \rightarrow \infty}\prod_{j=1}^{n}(d_{j}')^{f_{j}(\epsilon|_{E},\sigma)} \right |\\
        &{\color{white}tttttttttttttttt}+\left | \left ( \frac{p}{d_{1} \cdots d_{n}}\right )^{\sharp(\sigma)} - c^{\sharp(\sigma)}\right | {D(\epsilon|_{E},\sigma)}\\
        &\le c^{\sharp (\sigma)}\cdot \frac{1}{\mu(\mathbf{d})} +   \left | \frac{p}{d_{1} \cdots d_{n}}-c\right | (m+1)! \left ( \frac{p}{d_{1} \cdots d_{n}}+c+1\right)^{2m}\\
        &\leq \left (\frac{1}{\mu(\mathbf{d})}+ \left | \frac{p}{d_{1} \cdots d_{n}}-c\right | \right ) (m+1)! \left ( \frac{p}{d_{1} \cdots d_{n}}+c+1\right )^{2m},
    \end{align}
and this implies the desired conclusion
\begin{align}
    &\left | (\mathbb{E} \otimes \text{tr}) (a_{\epsilon(1)} \cdots a_{\epsilon(m)}) \right |\\
    &\leq \sum_{E\subseteq [m]}c^{m-l} \sum_{\sigma\in S_l} \left (\frac{1}{\mu(\mathbf{d})}+ \left | \frac{p}{d_{1} \cdots d_{n}}-c\right | \right ) (m+1)! \left ( \frac{p}{d_{1} \cdots d_{n}}+c+1\right )^{2m}\\
    &\leq 2^{m} m! \left (\frac{1}{\mu(\mathbf{d})}+ \left | \frac{p}{d_{1} \cdots d_{n}}-c\right | \right ) (m+1)! \left ( \frac{p}{d_{1} \cdots d_{n}}+c+1\right )^{3m}.
\end{align}
\end{proof}

Finally, we are ready to establish a central limit theorem for partial transposes by applying Lemma \ref{singleton}, Lemma \ref{double}, and Lemma \ref{freeness}.

\begin{theorem}\label{pCLT}
    Let $p=p(\mathbf{d})$ and $n=n(\mathbf{d})$ be $\n$-valued functions of $\mathbf{d}=(d_{1},\cdots,d_{n})$, and consider a sequence of subsets $B_{\mathbf{d}}\subseteq \left \{0,1\right\}^{n(\mathbf{d})}$. If 
    \begin{equation}
        \lim |B_{\bf d}|^m \left ( \frac{1}{\mu(\bf d)}+\left | \frac{p}{d_1d_2\cdots d_n}-c\right | \right )=0    
    \end{equation}
    for all natural numbers $m$ and $\lim |B_{\bf d}|=\infty$, then the following random matrices
    \begin{align}
        s_{\mathbf{d}}=\frac{1}{\sqrt{|B_{\mathbf{d}}|}}\sum_{x\in B_{\mathbf{d}}}\left ( W^x_{d_{1} \cdots d_{n},p}-c\cdot \text{Id}_{d_{1} \cdots d_{n}} \right )
    \end{align}
    converge in moments to the semicircular element of the mean $0$ and the variance $c$, i.e. we have
    \begin{align}
        \lim (\mathbb{E}\otimes \text{tr})(s_{\mathbf{d}}^{m})=\int_{[-2c,2c]}\frac{t^m}{2\pi c^2}\sqrt{4c^2-t^2}dt.
    \end{align}

\end{theorem}

\begin{proof}

    It is enough to prove that
    \begin{align}
        \lim (\mathbb{E}\otimes \text{tr})(s_{\mathbf{d}}^{m})=c^{\frac{m}{2}}  |NC_{2}(m)|
    \end{align}
    where $NC_2(m)$ is the set of all non-crossing pairings on $[m]$. For any $m$ and $\mathbf{d}$, we have
    \begin{align}
        (\mathbb{E}\otimes \text{tr})(s_{\mathbf{d}}^{m})&= \frac{1}{\sqrt{|B_{\mathbf{d}}|^{m}}} \sum_{x:[m]\rightarrow B_{\mathbf{d}}} (\mathbb{E} \otimes \text{tr}) (a_{\mathbf{d},x(1)} \cdots a_{\mathbf{d},x(m)}) \\
        &= \frac{1}{\sqrt{|B_{\mathbf{d}}|^{m}}} \sum_{\pi \in P(m)} \underset{\ker(x)=\pi}{\sum_{x:[m]\rightarrow B_{\mathbf{d}} :}} (\mathbb{E} \otimes \text{tr}) (a_{\mathbf{d},x(1)} \cdots a_{\mathbf{d},x(m)}).
    \end{align}
    Now let us take a representative function $x_{\mathbf{d},\pi}:[m]\rightarrow B_\mathbf{d}$ satisfying $\text{ker}(x_{\mathbf{d},\pi})=\pi$ for each $\mathbf{d}=(d_{1}, \cdots, d_{n})$ and $\pi\in P(m)$. Then we have \small
    \begin{align}
        &\left| \sum_{\substack{x:[m]\rightarrow B_{\mathbf{d}}\\\text{ker}(x)=\pi }}(\mathbb{E} \otimes \text{tr}) \left( \prod_{i=1}^{m} a_{\mathbf{d},x(i)}\right) - k_{\mathbf{d},\pi}\cdot  (\mathbb{E} \otimes \text{tr}) \left( \prod_{i=1}^{m} a_{\mathbf{d},x_{\mathbf{d},\pi}(i)} \right)\right|\\
        &\leq  \frac{2^{m+1} k_{\mathbf{d},\pi} m!  (1+c)^{m}}{\mu(\mathbf{d})}  \sum_{s=0}^{m} \left ( \frac{p}{d_{1} \cdots d_{n}} \right )^{s}\\
        &\leq \frac{2^{m+1} |B_{\mathbf{d}}|^{m} m!  (1+c)^{m}}{\mu(\mathbf{d})}  \sum_{s=0}^{m} \left ( \frac{p}{d_{1} \cdots d_{n}} \right )^{s}
    \end{align}
    \normalsize by Lemma \ref{welldef}, where $k_{\mathbf{d},\pi}=|B_{\mathbf{d}}|\cdot (|B_{\mathbf{d}}|-1)\cdots (|B_{\mathbf{d}}|-\sharp(\pi)+1)$. Thus, the given condition $\displaystyle |B_{\mathbf{d}}|^m=o(\mu(\mathbf{d}))$ implies
    \begin{align}
        &\lim (\mathbb{E} \otimes \text{tr}) (s_{\mathbf{d}}^{m})\\
        &= \sum_{\pi\in P(m)} \lim |B_{\mathbf{d}}|^{-\frac{m}{2}} k_{\mathbf{d},\pi}(\mathbb{E} \otimes \text{tr}) (a_{\mathbf{d},x_{\mathbf{d},\pi}(1)} \cdots a_{\mathbf{d},x_{\mathbf{d},\pi}(m)}).
    \end{align}

    As the first step, let us prove that
    \begin{equation}
        \lim |B_{\mathbf{d}}|^{-\frac{m}{2}} k_{\mathbf{d},\pi}(\mathbb{E} \otimes \text{tr}) (a_{\mathbf{d},x_{\mathbf{d},\pi}(1)} \cdots a_{\mathbf{d},x_{\mathbf{d},\pi}(m)})=0
    \end{equation}
    for the following situation
    \begin{center}
        ({\textbf{Case A}}) the partition $\pi$ contains a singleton block.
    \end{center} 
    Indeed, Lemma \ref{singleton} provides the following estimate
    \begin{align}
        &|B_{\mathbf{d}}|^{-\frac{m}{2}} k_{\mathbf{d},\pi}\left | (\mathbb{E} \otimes \text{tr}) (a_{\mathbf{d},x_{\mathbf{d},\pi}(1)} \cdots a_{\mathbf{d},x_{\mathbf{d},\pi}(m)}) \right |\\
        & \leq 2^{m}|B_{\mathbf{d}}|^{\frac{m}{2}}(1+c)^{m}m!\left (\frac{1}{\mu(\mathbf{d})}+\left | \frac{p}{d_{1} \cdots d_{n}}-c \right | \right )\sum_{s=0}^m \left ( \frac{p}{d_{1} \cdots d_{n}} \right )^s,
    \end{align}
    and the given conditions imply
    \begin{align}
        &\lim |B_{\mathbf{d}}|^{\frac{m}{2}}\left (\frac{1}{\mu(\mathbf{d})}+\left | \frac{p}{d_{1} \cdots d_{n}}-c \right | \right )\\
        &\leq \lim |B_{\mathbf{d}}|^{m}\left (\frac{1}{\mu(\mathbf{d})}+\left | \frac{p}{d_{1} \cdots d_{n}}-c \right | \right )=0.
    \end{align}
    
    From now on, it is enough to suppose that the permutation $\pi$ does not contain a singleton set, implying $\sharp(\pi) \le \frac{m}{2}$. Furthermore, if we suppose that $\pi$ is in ({\bf Case B}) i.e. $\sharp(\pi)<\frac{m}{2}$, then it is straightforward to see that
    \begin{align}
        &\lim |B_{\mathbf{d}}|^{-\frac{m}{2}} k_{\mathbf{d},\pi}(\mathbb{E} \otimes \text{tr}) (a_{\mathbf{d},x_{\mathbf{d},\pi}(1)} \cdots a_{\mathbf{d},x_{\mathbf{d},\pi}(m)})\\
        & \leq \lim |B_{\mathbf{d}}|^{\sharp(\pi)-\frac{m}{2}} (\mathbb{E} \otimes \text{tr}) (a_{\mathbf{d},x_{\mathbf{d},\pi}(1)} \cdots a_{\mathbf{d},x_{\mathbf{d},\pi}(m)})\\
        &=0 \cdot \lim (\mathbb{E} \otimes \text{tr}) (a_{\mathbf{d},x_{\mathbf{d},\pi}(1)} \cdots a_{\mathbf{d},x_{\mathbf{d},\pi}(m)}) = 0.
    \end{align}

    Thus, let us focus on the cases where $\pi$ is a pairing, i.e. all disjoint blocks of $\pi$ are given by cycles of length $2$. In this case, we have $\sharp(\pi) = \frac{m}{2}$ and the representative function $x_{\mathbf{d},\pi}:[m]\rightarrow B_{\mathbf{d}}$ should be one of the following two cases:
    \begin{itemize}
        \item ({\textbf{Case C}}) $\text{ker}(x_{\mathbf{d},\pi})$ is a pairing, and there exists $i\in [m-1]$ such that $\left\{i,i+1\right\}\in \text{ker}(x_{\mathbf{d},\pi})$
        \item ({\textbf{Case D}}) $\text{ker}(x_{\mathbf{d},\pi})$ is a pairing, and $x_{\mathbf{d},\pi}(i)\neq x_{\mathbf{d},\pi}(i+1)$ for all $i\in [m-1]$.
    \end{itemize}

    If $\pi$ is in ({\textbf{Case D}}), i.e. $\text{ker}(x_{\mathbf{d},\pi})$ is a pairing satisfying
    \begin{equation}
        x_{\mathbf{d},\pi}(i)\neq x_{\mathbf{d},\pi}(i+1)
    \end{equation}
    for all $i\in [m-1]$, then we have \small
    \begin{align}
        &\lim |B_{\mathbf{d}}|^{-\frac{m}{2}}k_{\mathbf{d},\pi}\left| (\mathbb{E} \otimes \text{tr}) \left( \prod_{i=1}^{m} a_{\mathbf{d},x_{\mathbf{d},\pi}(i)} \right) \right|\\
        &\leq \lim 2^{m} m! |B_{\mathbf{d}}|^{\frac{m}{2}}  \left (\frac{1}{\mu(\mathbf{d})}+ \left | \frac{p}{d_{1} \cdots d_{n}}-c\right | \right ) (m+1)! \left ( \frac{p}{d_{1} \cdots d_{n}}+c+1\right )^{3m}\\
        &\leq 2^{m} m!  (m+1)! \lim |B_{\mathbf{d}}|^{m}  \left (\frac{1}{\mu(\mathbf{d})}+ \left | \frac{p}{d_{1} \cdots d_{n}}-c\right | \right ) \left ( \frac{p}{d_{1} \cdots d_{n}}+c+1\right )^{3m}\\
        &=0
    \end{align}
     \normalsize by Lemma \ref{freeness}. Now, for the last situation ({\textbf{Case C}}), there exists $i_{0}\in [m-1]$ such that $x_{\mathbf{d},\pi}(i_0)=x_{\mathbf{d},\pi}(i_0+1)$ and Lemma \ref{double} implies
    \begin{align}
        &\lim |B_{\mathbf{d}}|^{-\frac{m}{2}}k_{\mathbf{d},\pi} (\mathbb{E} \otimes \text{tr}) \left( \prod_{i=1}^{m} a_{\mathbf{d},x_{\mathbf{d},\pi}(i)} \right)\\
        & = c\cdot \lim (\mathbb{E} \otimes \text{tr}) \left( \prod_{i \in [m] \backslash \{i_{0}, i_{0}+1 \}} a_{\mathbf{d},x_{\mathbf{d},\pi}(i)} \right).
    \end{align}
    Note that the restricted function $x_{\mathbf{d},\pi}|_{[m]\setminus \left\{i_0,i_0+1\right\}}$ defines a new pairing on $[m-2]$, which should be in one of ({\textbf{Case C}}) and ({\textbf{Case D}}). Thus, we can repeat the above arguments, leading us to conclude that 
    \begin{align}
        \lim |B_{\mathbf{d}}|^{-\frac{m}{2}}k_{\mathbf{d},\pi} (\mathbb{E} \otimes \text{tr}) \left( \prod_{i=1}^{m} a_{\mathbf{d},x_{\mathbf{d},\pi}(i)} \right)=
        \begin{cases}
            0, &\mbox{if } \pi \notin NC_{2}(m)\\
            c^{\frac{m}{2}}, &\mbox{if } \pi \in NC_{2}(m)
        \end{cases}.
    \end{align}
    Now, combining all the above discussions, we obtain
    \begin{align}
        &\lim (\mathbb{E} \otimes \text{tr}) (s_{\mathbf{d}}^{m}) \\
        &=\sum_{\pi \in P(m)} \lim |B_{\mathbf{d}}|^{-\frac{m}{2}}k_{\mathbf{d},\pi} (\mathbb{E} \otimes \text{tr}) \left( \prod_{i=1}^{m} a_{\mathbf{d},x_{\mathbf{d},\pi}(i)} \right)\\
        &= \sum_{\pi \in NC_{2}(m)} c^{\frac{m}{2}}= c^{\frac{m}{2}}|NC_{2}(m)|.
    \end{align}
\end{proof}

\emph{Acknowledgements}: The second author expresses gratitude to Professor James A. Mingo for the discussion at Queen's University, which was the starting point of this project. The authors were supported by Samsung Science and Technology Foundation under Project Number SSTF-BA2002-01 and by the National Research Foundation of Korea (NRF) grant funded by the Ministry of Science and ICT (MSIT) (No. 2020R1C1C1A01009681).

\appendix

\section{Proof of Theorem \ref{thm302}}

Let us begin with generalizing \cite[Lemma 3.2]{MP19} to the multipartite situation. More precisely, let us explain how to write the random variable 
\begin{equation}
    X_{\epsilon}=\text{tr}(W^{\epsilon_1}W^{\epsilon_2}\cdots W^{\epsilon_m}).
\end{equation}
as a polynomial of Gaussian variables for arbitrary $\z_2$-valued $m\times n$ matrices $\epsilon=(\epsilon_{ij})_{i\in [m],j\in [n]}$. Note that any $\epsilon=(\epsilon_{ij})_{i\in [m],j\in [n]}$ can be decomposed to $\epsilon'_{[n-1]}=(\epsilon_{ij})_{i\in [m],j\in [n-1]}$ and $\epsilon'_n=(\epsilon_{in})_{i\in [m]}$. Then $\epsilon'_{[n-1]}$ and $\epsilon'_n$ define the associated functions 
\begin{center}
    $\mathcal{E}_{[n-1]}:[n-1]\times [\pm m]\rightarrow [n-1]\times [\pm m]$ and $\mathcal{E}_n:[\pm m]\rightarrow [\pm m]$,
\end{center} 
where $\mathcal{E}_{[n-1]}$ is given by $ \mathcal{E}_{[n-1]}(j,x) = (j,\mathcal{E}_j(x))$.

\begin{notation}
We denote by $A(\epsilon)$ the set of all functions $\iota:[n-1]\times [\pm m]\rightarrow \cup_{j=1}^{n-1} [d_j]$ satisfying 
    \begin{enumerate}
        \item $\iota(j,\cdot)\in [d_j]$ for each $j\in [n-1]$, 
        \item $ \iota=\iota\circ\mathcal{E}_{[n-1]}\left (\text{id}_{n-1}\times \Gamma\Delta\Gamma^{-1}\right ) \mathcal{E}_{[n-1]}$,
    \end{enumerate} 
and by $B(\epsilon)$ the set of all functions $q:[\pm m]\rightarrow [d_{n}]$ satisfying
\begin{equation}
    q=q\circ\mathcal{E}_{n}\Gamma\Delta\Gamma^{-1}\mathcal{E}_{n}.
\end{equation} 
\end{notation}

Using the notations above, we explain how to express
\begin{equation}
    X_{\epsilon}=\text{tr}(W^{\epsilon_1}W^{\epsilon_2}\cdots W^{\epsilon_m}).
\end{equation}
as a polynomial of Gaussian variables in the following Lemma, which directly generalizes \cite[Lemma 3.2]{MP19} to the multipartite setting.

\begin{lemma}\label{lem200}
   Let $\epsilon=(\epsilon_{ij})_{i\in [m],j\in [n]}$ be a $\z_2$-valued $m\times n$ matrix with $\epsilon_i=(\epsilon_{ij})_{j=1}^{n}\in \left\{0,1\right\}^{n}$ for all $i\in [m]$, and let $X_{\epsilon}=\text{tr}(W^{\epsilon_1}W^{\epsilon_2}\cdots W^{\epsilon_m})$. Then we have
        \begin{align}
        \label{eq33}\left(d_{1} \cdots d_{n} \right)^{m+1} X_{\epsilon} =\sum_{\iota\in A(\epsilon)}\sum_{q\in B(\epsilon)} \sum_{t:[m]\rightarrow [p]} \prod_{y=1}^m g^{\iota_{y}}_{q(y),t(y)} \overline{g^{\iota_{-y}}_{q(-y),t(y)}}.
    \end{align}

\end{lemma}

\begin{proof}
   As the first step, in order to focus on an entrywise expression of 
    \begin{align}
       \notag &( d_{1} \cdots d_{n})W^{\epsilon_{i}}= \sum_{\mathbf{i},\mathbf{j} \in [d_{1} \cdots d_{n-1}]} 
       \left [ \bigotimes_{x=1}^{n-1} T^{\epsilon_{ix}}(e_{\mathbf{i}_{x},\mathbf{j}_{x}}) \right ] \otimes T^{\epsilon_{in}} (G_{\mathbf{i}}G_{\mathbf{j}}^{*}),
    \end{align}
let us introduce two functions ${\bf k}={\bf k}_{\bf i, j}$ and $\eta_i$ as follows.
\begin{itemize}
    \item Let us identify the following set
    \begin{align}
         F(n-1)&=[d_{n-1}]\times \cdots \times [d_1]\times [d_1]\times \cdots \times [d_{n-1}]\\
        &\cong [d_1d_2\cdots d_{n-1}]\times [d_1d_2\cdots d_{n-1}]
    \end{align}
    with the set of all functions ${\bf k}:[\pm (n-1)]\rightarrow \cup_{j=1}^{n-1} [d_j]$ satisfying that ${\bf k}(\pm t)\in [d_t]$ for all $t \in [n-1]$. More specifically, each pair $({\bf i},{\bf j})\in [d_1d_2\cdots d_{n-1}]\times [d_1d_2\cdots d_{n-1}]$ is associated with 
    \begin{align}
    {\bf k}&=({\bf k}(-(n-1)),\cdots,{\bf k}(-1),{\bf k}(1),\cdots,{\bf k}(n-1))\\
    &=({\bf j}_{n-1},\cdots,{\bf j}_{1},{\bf i}_1,\cdots,{\bf i}_{n-1}) \in F(n-1).
    \end{align}
    In this case, let us write ${\bf k}^+={\bf i}$ and ${\bf k}^-={\bf j}$ as functions from $[n-1]$ into $\cup_{j=1}^{n-1}[d_j]$.
\item For each $i\in [m]$, we define $\eta_{i}:[\pm (n-1)]\rightarrow [\pm (n-1)]$ by
\begin{equation}
    \eta_i(x) = (-1)^{\epsilon_{i|x|}}\cdot x.
\end{equation}
Then $\eta_{i}\circ\eta_{i}$ is the identity function on $[\pm (n-1)]$, and we have 
\begin{align}
     T^{\epsilon_{ij}}(e_{{\bf k}(j),{\bf k}(-j)})&=\left\{ \begin{array}{ll}
    e_{{\bf k}(j),{\bf k}(-j)}&\text{if }\epsilon_{ij}=0\\
    e_{{\bf k}(-j),{\bf k}(j)}&\text{if }\epsilon_{ij}=1
    \end{array} \right .\\
    \label{eq200}&= e_{({\bf k}\circ \eta_i)(j),({\bf k}\circ \eta_i)(-j)} ,
\end{align} 
where $T$ is the transpose operator.
\end{itemize}
Using the notations above, we obtain
\begin{align}
 & \notag ( d_{1} \cdots d_{n})W^{\epsilon_{i}}= \sum_{\mathbf{k}\in F(n-1)} \left [ \bigotimes_{x=1}^{n-1} T^{\epsilon_{ix}} \left ( e_{\mathbf{k}(x),\mathbf{k}(-x)}\right ) \right ]\otimes T^{\epsilon_{in}}\left ( G_{{\bf k}^+} G_{{\bf k}^{-}}^{*}\right )\\
   &=\sum_{\mathbf{k}\in F(n-1)} \left [ \bigotimes_{x=1}^{n-1} e_{(\mathbf{k}\circ\eta_{i})(x),(\mathbf{k}\circ\eta_{i})(-x)}  \right ] \otimes T^{\epsilon_{in}} \left ( G_{{\bf k}^+}G_{{\bf k}^{-}}^{*}\right )
\end{align}
by \eqref{eq200}. Furthermore, since ${\bf k}\mapsto {\bf k}\circ \eta_i$ is a bijective function on $F(n-1)$ and $\eta_i\circ \eta_i=\text{id}_{[\pm (n-1)]}$, we have
\begin{align}
( d_{1} \cdots d_{n})W^{\epsilon_{i}}=\sum_{\mathbf{k}\in F(n-1)} \left [ \bigotimes_{x=1}^{n-1} e_{\mathbf{k}(x),\mathbf{k}(-x)} \right ]\otimes T^{\epsilon_{in}} \left ( G_{({\bf k}\circ \eta_i)^+}G_{({\bf k}\circ \eta_i)^-}^{*}\right ).
\end{align}

Thus, the joint moment $\left( d_{1} \cdots d_{n+1} \right)^{m+1}X_{\epsilon}$ is written as \footnotesize
    \begin{align}\label{eq203}
 \sum_{\mathbf{k}_{1},\cdots,\mathbf{k}_{m}\in F(n-1)} \text{Tr}\left ( \bigotimes_{x=1}^{n-1}   \left [\prod_{y=1}^m e_{\mathbf{k}_{y}(x),\mathbf{k}_{y}(-x)}  \right ] \right )\cdot   \text{Tr}\left ( \prod_{y=1}^{m}T^{\epsilon_{yn}} (G_{(\mathbf{k}_{y}\circ\eta_{y})^+}G_{(\mathbf{k}_{y}\circ\eta_{y})^- }^{*})
       \right ),
    \end{align}
 \normalsize   and we have
    \begin{align}
        \label{eq201}&\text{Tr}\left ( \bigotimes_{x=1}^{n-1}   \left [\prod_{y=1}^m e_{\mathbf{k}_{y}(x),\mathbf{k}_{y}(-x)}  \right ] \right )=\prod_{x=1}^{n-1}\text{Tr}\left ( \prod_{y=1}^m e_{\mathbf{k}_{y}(x),\mathbf{k}_{y}(-x)} \right ).
    \end{align}

To deal with the multiple functions ${\bf k}_1$, ${\bf k}_2$, $\cdots$, ${\bf k}_m$ simultaneously, let us define $K_0:[m]\times [\pm (n-1)]\rightarrow \cup_{j=1}^{n-1}[d_j]$ by
\begin{equation}
    K_0(i,j)={\bf k}_i(j)\in [d_{|j|}],
\end{equation}
and its natural extension $K$ on $[\pm m]\times [\pm (n-1)]$ given by
\begin{equation}
    K(i,j)= {\bf k}_{|i|}(\text{sgn}(i)\cdot j)\in [d_{|j|}].
\end{equation}
Then it is straightforward to see that \eqref{eq201} is given by
\begin{align}
   &\prod_{x=1}^{n-1}\text{Tr}\left ( \prod_{y=1}^m e_{\mathbf{k}_{y}(x),\mathbf{k}_{y}(-x)} \right )\\
   &= \left\{\begin{array}{ll}
1&\text{if }K(-i,j)=K(\Gamma(i),j), i\in [m],j\in [n-1]\\
0&\text{otherwise}
    \end{array} \right .
\end{align}
where $\Gamma=(1,2,\cdots, m)\in S_m$. Now, let us define $F_0(m,n)$ as the set of $K_0=({\bf k}_1,{\bf k}_2,\cdots,{\bf k}_m)\in F(n-1)^m$ satisfying the condition 
\begin{equation}\label{eq30}
    K(-i,j)=K(\Gamma(i),j)\text{ for all }i\in [m]\text{ and }j\in [n-1].
\end{equation} 
Then the expression \eqref{eq203} is simplified to
   \begin{align}\label{eq203.5}
        & \sum_{(\mathbf{k}_{1},\cdots,\mathbf{k}_{m})\in F_0(m,n)} \text{Tr}\left ( \prod_{y=1}^{m}T^{\epsilon_{yn}} (G_{(\mathbf{k}_{y}\circ\eta_{y})^+}G_{(\mathbf{k}_{y}\circ\eta_{y})^- }^{*})
       \right ).
    \end{align}
 
On the other hand, any $({\bf k}_1,{\bf k}_2,\cdots,{\bf k}_m)\in F(n-1)^m$ is associated to a function $\iota: [n-1]\times [\pm m] \rightarrow  \cup_{j=1}^{n-1} [d_j]$ given by
\begin{equation}
    \iota(j,i)=K(i,\eta_{|i|}(j)) \in [d_j] .
\end{equation}
Indeed, the above condition \eqref{eq30} is equivalent to that $\iota\in A(\epsilon)$, i.e.
\begin{equation}
    \iota=\iota\circ\mathcal{E}_{[n-1]}\left (\text{id}_{n-1}\times \Gamma\Delta\Gamma^{-1}\right ) \mathcal{E}_{[n-1]}
\end{equation}
on $[n-1]\times [\pm m]$, and the restricted functions $\iota_y=\iota(\cdot, y)$ and $\iota_{-y}=\iota(\cdot, -y)$ satisfy
\begin{align}
&\iota_y(j)=K(y,\eta_y(j))={\bf k}_y(\eta_y(j))=({\bf k}_y\circ \eta_y)(j)\\
&\iota_{-y}(j)=K(-y,\eta_y(j))={\bf k}_y(-\eta_y(j))={\bf k}_y(\eta_y(-j))=({\bf k}_y\circ \eta_y)(-j)
\end{align}
for all $j\in [n-1]$ and $y\in [m]$. Thus, combining \eqref{eq203} and \eqref{eq203.5}, we have 
    \begin{align}
        \left( d_{1} \cdots d_{n} \right)^{m+1}X_{\epsilon} = \sum_{\iota\in A(\epsilon)} \text{Tr}\left ( \prod_{y=1}^m T^{\epsilon_{yn}} (G_{\iota_y}G_{\iota_{-y}}^{*})\right ).
    \end{align}
 Note that
    \begin{align}
        & T^{\epsilon_{yn}} \left ( G_{\iota_y} G_{\iota_{-y}}^{*}\right )\\
        &= T^{\epsilon_{yn}} \left(\sum_{r(y),r(-y)=1}^{d_{n}} \left [\sum_{t(y)=1}^{p} g^{\iota_y}_{r(y),t(y)} \overline{g^{\iota_{-y}}_{r(-y),t(y)}} \right ] e_{r(y),r(-y)}\right)\\
        &= \sum_{r(y),r(-y)=1}^{d_{n}} \left [ \sum_{t(y)=1}^{p} g^{\iota_y}_{r(y),t(y)} \overline{g^{\iota_{-y}}_{r(-y),t(y)}} \right ] e_{{(r\circ \mathcal{E}_{n})(y),(r\circ \mathcal{E}_{n})(-y)}}\\
        \label{eq31}&= \sum_{r(y),r(-y)=1}^{d_{n}} \left [ \sum_{t(y)=1}^{p} g^{\iota_y}_{(r\circ\mathcal{E}_{n})(y),t(y)} \overline{g^{\iota_{-y}}_{(r\circ\mathcal{E}_{n})(-y),t(y)}} \right ] e_{{r(y),r(-y)}}
    \end{align}
for any $y\in [m]$, and that the non-trivial terms of the trace of $\displaystyle \prod_{y=1}^m T^{\epsilon_{yn}}(G_{\iota_y}G_{\iota_{-y}}^*)$ arise only from the cases where we have
\begin{equation}\label{eq31.5}
    r(-1)=r(2),~r(-2)=r(3),~\cdots,r(-m)=r(1).
\end{equation}
Furthermore, \eqref{eq31.5} is also equivalent to that $q\in B(\epsilon)$, i.e. $q=r\circ\mathcal{E}_{n}:[\pm m]\rightarrow [d_{n}]$ satisfies 
\begin{equation}
    q=q\circ\mathcal{E}_{n}\Gamma\Delta\Gamma^{-1}\mathcal{E}_{n}.
\end{equation} 

Finally, combining all the discussions above, we obtain
    \begin{align}
        &\text{Tr}\left ( \prod_{y=1}^m T^{\epsilon_{yn}} (G_{\iota_y}G_{\iota_{-y}}^{*})\right )=\sum_{q\in B(\epsilon)}\sum_{t:[m]\rightarrow [p]} \prod_{y=1}^m g^{\iota_y}_{q(y),t(y)} \overline{g^{\iota_{-y}}_{q(-y),t(y)}},
    \end{align}
which leads us to the following conclusion
    \begin{align}
        &\left ( d_{1} \cdots d_{n} \right )^{m+1} X_{\epsilon} = \sum_{\iota\in A(\epsilon)}\sum_{q\in B(\epsilon)} \sum_{t:[m]\rightarrow [p]} \prod_{y=1}^m g^{\iota_y}_{q(y),t(y)} \overline{g^{\iota_{-y}}_{q(-y),t(y)}}.
    \end{align}

\end{proof}

A non-trivial fact from Lemma \ref{lem200} is that $X_{\epsilon}$ is a real-valued random variable and, moreover, the explicit expression \eqref{eq33} can be applied to compute the following $k$-th moments
\begin{equation}
   \mathbb{E}\left ( X_{\epsilon}^k \right )=\mathbb{E}\left ( \left [ \text{tr}(W^{\epsilon_1}W^{\epsilon_2}\cdots W^{\epsilon_m}) \right ]^k \right ).
\end{equation}

 Let $k$ be an arbitrary natural number and let $\epsilon=(\epsilon_{ij})_{i\in [m],j\in [n]}$ be a $\z_2$-valued $m\times n$ matrix with $\epsilon_i=(\epsilon_{ij})_{j\in [n]} \in \left\{0,1\right\}^n$. Let us apply the explicit formula \eqref{eq33} of 
 \begin{equation}
 d_1\cdots d_n X_{\epsilon}=\text{Tr}(W^{\epsilon_1}W^{\epsilon_2}\cdots W^{\epsilon_m})
 \end{equation}
to find a suitable expression of the $k$-th powers
   \begin{equation}\label{eq32}
        \left ( d_{1} \cdots d_{n} \right )^{(m+1)k} X_{\epsilon}^k=\left [ \sum_{\iota\in A(\epsilon)}\sum_{q\in B(\epsilon)} \sum_{t:[m]\rightarrow [p]} \prod_{y=1}^m g^{\iota_y}_{q(y),t(y)} \overline{g^{\iota_{-y}}_{q(-y),t(y)}} \right ]^k
    \end{equation}
where we need to consider the following multiple choices of functions:
\begin{equation}
    \iota^{(s)}\in A(\epsilon), ~q^{(s)}\in B(\epsilon), ~t^{(s)}:[m]\rightarrow [p]~ (1\leq s\leq k).
\end{equation}
To deal with all these functions simultaneously, let us introduce the following three multivariate functions:
\begin{itemize}
    \item ${\bf I}: [k]\times [\pm m]\rightarrow [d_1d_2\cdots d_{n-1}]$ given by
    \begin{equation}
        {\bf I}(s,s')=(I_1(s,s'),I_2(s,s'),\cdots,I_{n-1}(s,s')).
    \end{equation}
    and each $I_j:[k]\times [\pm m]\rightarrow [d_j]$ is given by
        \begin{equation}
        I_j(s,s')=\iota^{(s)}(j,s')\in [d_j].
        \end{equation}
    Then all $\iota^{(1)}$, $\iota^{(2)}$, $\cdots$, $\iota^{(k)}$ are in $A(\epsilon)$ if and only if 
    \begin{equation}
    I_{j}=I_{j}\circ \mathcal{E}^{(k)}_{j} \Gamma^{(k)} \Delta^{(k)} (\Gamma^{(k)})^{-1} \mathcal{E}^{(k)}_{j}    
    \end{equation}
    for all $j\in[n-1]$. We denote by $A(\epsilon,k)$ the set of such functions $\mathbf{I}$.
    \item ${\bf Q}:[k]\times [\pm m]\rightarrow [d_{n}]$ given by
    \begin{equation}
        {\bf Q}(s,s')=q^{(s)}(s').
    \end{equation}
    Then all $q^{(1)}$, $q^{(2)}$, $\cdots$, $q^{(k)}$ are in $B(\epsilon)$ if and only if
    \begin{equation}
        {\bf Q}={\bf Q}\circ \mathcal{E}^{(k)}_{n} \Gamma^{(k)} \Delta^{(k)} (\Gamma^{(k)})^{-1} \mathcal{E}^{(k)}_{n}.
    \end{equation}
 Let us denote by $B(\epsilon,k)$ the set of such functions ${\bf Q}$.
    
    \item ${\bf T}:[k]\times [m]\rightarrow [p]$ given by
    \begin{equation}
        {\bf T}(s,s')=t^{(s)}(s').
    \end{equation}
\end{itemize}
Note that, for each $(s,s')\in [k]\times [\pm m]$, the above ${\bf I}(s,s')$ can be considered a function from $[n-1]$ into $\cup_{j=1}^{n-1}[d_j]$ satisfying 
\begin{equation}
[{\bf I}(s,s')](j)=I_j(s,s')\in [d_j].    
\end{equation}
Then all our discussions are summarized into the following form:
\begin{align}
        \notag &\left(d_{1} \cdots d_{n} \right)^{(m+1)k} X_{\epsilon}^{k} \\
        \label{lemmaxk}&=\sum_{{\bf I}\in A(\epsilon,k)}\sum_{{\bf Q}\in B(\epsilon,k)}\sum_{{\bf T}:[k]\times [m]\rightarrow [p]}
        \prod_{s=1}^k \prod_{s'=1}^m g^{\mathbf{I}(s,s')}_{\mathbf{Q}(s,s'),\mathbf{T}(s,s')} \overline{g^{\mathbf{I}(s,-s')}_{\mathbf{Q}(s,-s'),\mathbf{T}(s,s')}}
\end{align}

Now, let us present a proof of Theorem \ref{thm302} using the above notations. Our arguments are analogous to the proof of \cite[Theorem 3.7]{MP19}.

\begin{proof}[Proof of Theorem \ref{thm302}]
    By \eqref{lemmaxk}, we have
    \begin{align}
        &\left(d_{1} \cdots d_{n} \right)^{(m+1)k} \mathbb{E}(X_{\epsilon}^{k}) \\
        &=\sum_{{\bf I}\in A(\epsilon,k)}\sum_{{\bf Q}\in B(\epsilon,k)}\sum_{{\bf T}:[k]\times [m]\rightarrow [p]} \mathbb{E}\left(\prod_{s=1}^k \prod_{s'=1}^m g^{\mathbf{I}(s,s')}_{\mathbf{Q}(s,s'),\mathbf{T}(s,s')} \overline{g^{\mathbf{I}(s,-s')}_{\mathbf{Q}(s,-s'),\mathbf{T}(s,s')}}\right).\label{exkfirst}
    \end{align}
    For any $\mathbf{I}\in A(\epsilon,k)$, $\mathbf{Q}\in B(\epsilon,k)$, and $\mathbf{T}:[k]\times[m]\rightarrow[p]$, let us write 
    \begin{align}
    &g_{\alpha(s,s')}=g^{\mathbf{I}(s,s')}_{\mathbf{Q}(s,s'),\mathbf{T}(s,s')}\\
    &g_{\beta(s,s')}=g^{\mathbf{I}(s,-s')}_{\mathbf{Q}(s,-s'),\mathbf{T}(s,s')}
    \end{align}
     for all $(s,s')\in [k]\times [m]\cong [km]$ to pursue simplicity.
    Then we have
    \begin{align}
        &\mathbb{E}\left(\prod_{s=1}^k \prod_{s'=1}^m g^{\mathbf{I}(s,s')}_{\mathbf{Q}(s,s'),\mathbf{T}(s,s')} \overline{g^{\mathbf{I}(s,-s')}_{\mathbf{Q}(s,-s'),\mathbf{T}(s,s')}}\right)\\
        \label{eg}&=\mathbb{E}\left(\prod_{x\in [k]\times [m]}g_{\alpha(x)}\cdot\prod_{y\in [k]\times [m]}\overline{g_{\beta(y)}}\right) = |\{\sigma\in S_{km}:\beta=\alpha\circ\sigma\}|,
    \end{align}
where the second equality comes from the Wick formula. Let us denote by $C(\sigma)$ the set of all triples $(\mathbf{I},\mathbf{Q},\mathbf{T})$ satisfying $\beta=\alpha\circ \sigma$ for $\sigma\in S_{km}$. Then we have
    \begin{align}
        \notag&\left(d_{1} \cdots d_{n} \right)^{(m+1)k} \mathbb{E}(X_{\epsilon}^{k})\\
        &=\sum_{\mathbf{I}, \mathbf{Q}, \mathbf{T}} \sum_{\sigma\in S_{km}:~\beta=\alpha\circ \sigma}1=\sum_{\sigma\in S_{km}}\sum_{(\mathbf{I}, \mathbf{Q}, \mathbf{T})\in C(\sigma)} 1=\sum_{\sigma\in S_{km}} |C(\sigma)|.
    \end{align}

    Using the natural identification $[\pm km]\cong [k]\times[\pm m]$, we can regard maps from $[k]\times[\pm m]$ as maps from $[\pm km]$.
    Also bijections on $[k]\times[\pm m]$ can be regarded as permutations on $[\pm km]$. Then $(\mathbf{I}, \mathbf{Q}, \mathbf{T})\in C(\sigma)$ if and only if the following conditions hold:
    \begin{itemize}
        \item [(A)] $I_{j}=I_{j}\circ\sigma\Delta^{(k)}\sigma^{-1}$ for all $j\in [n-1]$,
        \item [(B)] $\mathbf{Q}=\mathbf{Q}\circ\sigma\Delta^{(k)}\sigma^{-1}$.
        \item [(C)] $\mathbf{T}=\mathbf{T}\circ\sigma$.
    \end{itemize}

Since the conditions $\mathbf{I}\in A(\epsilon,k)$ and $\mathbf{Q}\in B(\epsilon,k)$ should be taken into account, $|C(\sigma)|$ is equal to the number of triples $(\mathbf{I}, \mathbf{Q}, \mathbf{T})$ consisting of general functions ${\bf I}:[k]\times [\pm m]\rightarrow [d_1d_2\cdots d_{n-1}]$, ${\bf Q}:[k]\times [\pm m]\rightarrow [d_n]$, ${\bf T}:[k]\times [m]\rightarrow [p]$ satisfying
\begin{itemize}
       \item [(A')] $I_{j} =I_{j}\circ\mathcal{E}^{(k)}_{j} \Gamma^{(k)} \Delta^{(k)} (\Gamma^{(k)})^{-1} \mathcal{E}^{(k)}_{j}=I_{j}\circ\sigma\Delta^{(k)}\sigma^{-1}$ for all $j\in [n-1]$,
        \item [(B')] $\mathbf{Q} =\mathbf{Q}\circ\mathcal{E}^{(k)}_{n} \Gamma^{(k)} \Delta^{(k)} (\Gamma^{(k)})^{-1} \mathcal{E}^{(k)}_{n}=\mathbf{Q}\circ\sigma\Delta^{(k)}\sigma^{-1}$.
        \item [(C')] $\mathbf{T}=\mathbf{T}\circ\sigma$.
\end{itemize}

    Thus the number of such triples $(\mathbf{I}, \mathbf{Q}, \mathbf{T})$ is given by
    \begin{align}
        \left[ \prod_{j=1}^{n} d_{j}^{\sharp(\mathcal{E}^{(k)}_{j} \Gamma^{(k)} \Delta^{(k)} (\Gamma^{(k)})^{-1} \mathcal{E}^{(k)}_{j} \lor \sigma\Delta^{(k)}\sigma^{-1})} \right] \cdot p^{\sharp(\sigma)},
    \end{align}
which leads us to the following identity
    \begin{align}
        &\left(d_{1} \cdots d_{n} \right)^{(m+1)k} \mathbb{E}(X_{\epsilon}^{k})\\
        &=\sum_{\sigma \in S_{km}} \left[ \prod_{j=1}^{n} d_{j}^{\sharp(\mathcal{E}^{(k)}_{j} \Gamma^{(k)} \Delta^{(k)} (\Gamma^{(k)})^{-1} \mathcal{E}^{(k)}_{j} \lor \sigma\Delta^{(k)}\sigma^{-1})} \right] \cdot p^{\sharp(\sigma)}.
    \end{align}
This implies the following desired conclusion
    \begin{align}
        &\mathbb{E}(X_{\epsilon}^{k})=\sum_{\sigma \in S_{km}} \left( \frac{p}{d_{1} \cdots d_{n}} \right)^{\sharp(\sigma)} \prod_{j=1}^{n} d_{j}^{f_{k,j}(\epsilon,\sigma)}.
    \end{align}
\end{proof}

\bibliographystyle{alpha}
\bibliography{bib}

\begin{thebibliography}{CMHW19}

\bibitem[Bra96]{Br96}
Samuel~L. Braunstein.
\newblock Geometry of quantum inference.
\newblock {\em Phys. Lett. A}, 219(3-4):169--174, 1996.

\bibitem[Cho82]{Ch80}
M.-D. Choi.
\newblock Positive linear maps, {O}perator {A}lgebras and {A}pplications
  ({K}ingston, 1980).
\newblock {\em Proc. Sympos. Pure Math.}, 38 Part 2, {A}mer. {M}ath.
  {S}oc.:583--590, 1982.

\bibitem[Chr]{Ch12}
M.~Christandl.
\newblock {PPT} square conjecture. {B}anff international {R}esearch {S}tation
  {W}orkshop: {O}perator {S}tructures in {Q}uantum {I}nformation {T}heory
  (2012).

\bibitem[CMHW19]{CMW19}
Matthias Christandl, Alexander M\"{u}ller-Hermes, and Michael~M. Wolf.
\newblock When do composed maps become entanglement breaking?
\newblock {\em Ann. Henri Poincar\'{e}}, 20(7):2295--2322, 2019.

\bibitem[CYT19]{CYT19}
Lin Chen, Yu~Yang, and Wai-Shing Tang.
\newblock Positive-partial-transpose square conjecture for $n=3$.
\newblock {\em Phys. Rev. A}, 99:012337, Jan 2019.

\bibitem[Dyk93]{Dy93}
Ken Dykema.
\newblock On certain free product factors via an extended matrix model.
\newblock {\em J. Funct. Anal.}, 112(1):31--60, 1993.

\bibitem[Hal98]{Ha98}
Michael J.~W. Hall.
\newblock Random quantum correlations and density operator distributions.
\newblock {\em Phys. Lett. A}, 242(3):123--129, 1998.

\bibitem[HHH96]{HHH96}
Micha{\l} Horodecki, Pawe{\l} Horodecki, and Ryszard Horodecki.
\newblock Separability of mixed states: necessary and sufficient conditions.
\newblock {\em Phys. Lett. A}, 223(1-2):1--8, 1996.

\bibitem[KMP18]{KMP18}
Matthew Kennedy, Nicholas~A. Manor, and Vern~I. Paulsen.
\newblock Composition of {PPT} maps.
\newblock {\em Quantum Inf. Comput.}, 18(5-6):472--480, 2018.

\bibitem[MP13]{MP13}
James~A. Mingo and Mihai Popa.
\newblock Real second order freeness and {H}aar orthogonal matrices.
\newblock {\em J. Math. Phys.}, 54(5):051701, 35, 2013.

\bibitem[MP19]{MP19}
James~A. Mingo and Mihai Popa.
\newblock Freeness and the partial transposes of {W}ishart random matrices.
\newblock {\em Canad. J. Math.}, 71(3):659--681, 2019.

\bibitem[MP22]{MP22}
James~A. Mingo and Mihai Popa.
\newblock The partial transpose and asymptotic free independence for {W}ishart
  random matrices, {II}.
\newblock {\em Pacific J. Math.}, 317(2):387--421, 2022.

\bibitem[MS17]{MS17}
James~A. Mingo and Roland Speicher.
\newblock {\em Free Probability and Random Matrices}, volume~35 of {\em Fields
  Institute Monographs}.
\newblock Springer, New York, NY, 2017.

\bibitem[Nic93]{Ni93}
Alexandru Nica.
\newblock Asymptotically free families of random unitaries in symmetric groups.
\newblock {\em Pacific J. Math.}, 157(2):295--310, 1993.

\bibitem[NS06]{NS06}
Alexandru Nica and Roland Speicher.
\newblock {\em Lectures on the combinatorics of free probability}, volume 335
  of {\em London Mathematical Society Lecture Note Series}.
\newblock Cambridge University Press, Cambridge, 2006.

\bibitem[Per96]{Pe96}
Asher Peres.
\newblock Separability criterion for density matrices.
\newblock {\em Phys. Rev. Lett.}, 77(8):1413--1415, 1996.

\bibitem[RJP18]{RJP18}
Mizanur Rahaman, Samuel Jaques, and Vern~I. Paulsen.
\newblock Eventually entanglement breaking maps.
\newblock {\em J. Math. Phys.}, 59(6):062201, 11, 2018.

\bibitem[Sho02]{Sh02}
Peter~W. Shor.
\newblock Additivity of the classical capacity of entanglement-breaking quantum
  channels.
\newblock {\em J. Math. Phys.}, 43(9):4334--4340, 2002.

\bibitem[SS12]{SmSm12}
Graeme Smith and John~A. Smolin.
\newblock Detecting incapacity of a quantum channel.
\newblock {\em Phys. Rev. Lett.}, 108:230507, Jun 2012.

\bibitem[SZ04]{SZ04}
Hans-J\"{u}rgen Sommers and Karol \.{Z}yczkowski.
\newblock Statistical properties of random density matrices.
\newblock {\em J. Phys. A}, 37(35):8457--8466, 2004.

\bibitem[Voi91]{Vo91}
Dan Voiculescu.
\newblock Limit laws for random matrices and free products.
\newblock {\em Invent. Math.}, 104(1):201--220, 1991.

\bibitem[Voi98]{Vo98}
Dan Voiculescu.
\newblock A strengthened asymptotic freeness result for random matrices with
  applications to free entropy.
\newblock {\em Internat. Math. Res. Notices}, (1):41--63, 1998.

\bibitem[Wor76]{Wo76}
S.~L. Woronowicz.
\newblock Positive maps of low dimensional matrix algebras.
\newblock {\em Rep. Math. Phys.}, 10(2):165--183, 1976.

\bibitem[ZS01]{ZS01}
Karol \.{Z}yczkowski and Hans-J\"{u}rgen Sommers.
\newblock Induced measures in the space of mixed quantum states.
\newblock {\em J. Phys. A}, 34(35):7111--7125, 2001.

\end{thebibliography}

\end{document}